\pdfoutput=1
\documentclass[11pt, oneside]{article}      
\usepackage{geometry}                       
\usepackage{graphicx}               
\usepackage{amssymb}
\usepackage{amsthm}
\usepackage{cite}
\usepackage{amsmath}
\usepackage{mathtools}
\usepackage{enumerate}
\usepackage[scr=boondoxo,
scrscaled=0.98,
cal=esstix,
frak=esstix,
frakscaled=0.8
]{mathalfa}

\usepackage{hyperref}
\hypersetup{
    colorlinks=true,
    linkcolor=black,
    citecolor=black,
    filecolor=black,
    urlcolor=black,
}

\usepackage{scrextend}

\usepackage{xcolor}
\definecolor{myColor}{HTML}{e74c3c}

\usepackage{everysel}

\everymath{\color{myColor} \everymath{} }
\usepackage{etoolbox}
\AtBeginEnvironment{cases}{\everymath{}}
\AtBeginEnvironment{equation}{\everymath{}}

\let\originaldisplaystyle\displaystyle \renewcommand\displaystyle{\color{black}\originaldisplaystyle}

\newcommand{\proposition}[1]{
\begin{addmargin}[2em]{2em}
\vspace{5px}
#1
\vspace{5px}
\end{addmargin}}

\theoremstyle{definition}
\newtheorem{defn}{Definition}
\newtheorem{ax}{Axiom}

\newtheorem*{note}{Note}
\newtheorem*{summary}{Summary}
\newtheorem{lem}{Lemma}
\newtheorem{theo}{Theorem}

\newtheorem{proofpart}{Part}[theo]

\numberwithin{equation}{section}

\newtheorem*{remark}{Remark}

\usepackage{listings}
\usepackage{color}

\definecolor{dkgreen}{rgb}{0,0.6,0}
\definecolor{gray}{rgb}{0.5,0.5,0.5}
\definecolor{mauve}{rgb}{0.58,0,0.82}

\newcommand{\reff}[1]{(\ref{#1})}
\newcommand*{\Scale}[2][4]{\scalebox{#1}{$#2$}}%

\usepackage{accents}
\newcommand\undervec[1]{\underaccent{\vec}{#1}}

\newtheoremstyle{c}{}{}{}{}{\bfseries}{.}{.5em}{#1 to #3}
\theoremstyle{c}
\newtheorem*{prelude}{Prelude}

\title{A weak reduction of the Erd{\"o}s-Szekeres conjecture into a constraint unsatisfiability problem regarding certain multisets}
\author{\color{black} Archy Will He}
\date{7th Nov 2015}
\begin{document}
\maketitle
\begin{abstract}
We introduce the theory of \textit{div point sets}, which aims to provide a framework to study the combinatoric nature of any set of points in general position on an Euclidean plane. We then show that proving the unsatisfiability of some first-order logic formulae concerning some sets of multisets of uniform cardinality over boolean variables would prove the Erd{\"o}s-Szekeres conjecture, which states that for any set of $2^{n-2}+1$ points in general position, there exists $n$ points forming a convex polygon, where $n \geq 3$.
\end{abstract}

\section{Introduction}
In early 20th century Erd{\"o}s and Szekeres \cite{erdos1935} showed that for all $n \geq 3$, there exists some integer $N \geq n$ such that, among any $N$ points in general position on an Euclidean plane, there are $n$ points forming a convex polygon, and conjectured that the smallest value for such $N$ is determined by the function $g$ where $g(n)=2^{n-2}+1$. This is now known as the Erd{\"o}s-Szekeres conjecture (and the problem of determining the smallest $N$ is often referred to as the \textit{Happy Ending Problem} since it led to the marriage of Szekeres and Klein, who first proposed the question). In their second paper, Erd{\"o}s and Szekeres \cite{erdos1960} showed that $g(n)$ is certainly greater than $2^{n-2}$. Currently, the best known bounds are
$$2^{n-2}+1 \leq g(n) \leq {2n-5 \choose n-2} +1$$
Throughout the decades many improvements for the upper bound have been made. The current upper bound was obtained by T{\'o}th and Valtr \cite{toth1998} in 1998 as an improvement to the previous one by Kleitman and Pachter \cite{kleitman1998} in the same year.

In 2002, using an exhaustive search algorithm, Szekeres and Peters \cite{szekeres2006} were able to demonstrate that the conjecture holds for $n=6$. To this day it remains the greatest $n$ for which we know for certainty that the smallest $N$ is indeed $2^{n-2}+1$.

The aim of this article is to show that for every $n \geq 5$, there exists an instance of a constraint unsatisfiability problem, which, if solved (i.e. proving that some FOL propositions about certain multisets are unsatisfiable), would prove that the conjecture holds for $n$, through the theory of \textit{div point sets}.

\subsection{preliminary}
Throughout the article we would assume Zermelo-Fraenkel set theory (ZF). The word \textit{class} would be used to denote a collection of sets satisfying some predicate $\phi$. Everything would be formulated under first order logic (FOL). $\mathbb{N}_{\geq c}$ would be used to refer to the set of natural numbers greater or equal to some $c \in \mathbb{N}$. For any 2 natural numbers $a$, $b$, $\binom{a}{b}$ denotes the binomial coefficient $a \; choose \; b$. $\land$,$\lor$,$\neg$,$\Rightarrow$ and $\Leftrightarrow$ would be used to mean \textit{and}, \textit{or}, \textit{not}, \textit{imply} and \textit{iff} respectively. We write $A := B$ to mean $A$ is defined to be equivalent to $B$. $\forall x_1\in A \; \forall x_2 \in A \; \forall x_3 \in A \; ... \forall x_n \in A$ would be abbreviated to
\[
\forall x_1,x_2,x_3 ... x_n \in A
\]
and  $\exists x_1 \in A \; \exists x_2 \in A \; \exists x_3 \in A \; ... \exists x_n \in A $ to
\[
\exists x_1,x_2,x_3 ... x_n \in A
\]
For any set $V$, $|V|$ would be used to denote its cardinality. When a set $V$ has a cardinality of $k$, we may describe it as a \textit{$k$-cardinality set}. $\mathcal{P}(V)$ would be used to denote its power set.
The subscript of a set union or intersection may be omitted to indicate that the union or intersection is applied to each element in the set i.e.
\begin{align*}
&\text{For any set } A \\
& \bigcup A = \bigcup_{a \in A} a =  \bigcup_{k=1}^n a_k  = a_1 \cup a_2 \cup .. \cup a_n \\
& \bigcap A = \bigcap_{a \in A} a =  \bigcap_{k=1}^n a_k = a_1 \cap a_2 \cap .. \cap a_n \\
&\text{where } |A| = n \text{ and } a_1, a_2 ... a_n \text{ are unique elements of } A
\end{align*}
This use of notation applies to $\bigvee$ and $\bigwedge$ as well:
\begin{align*}
&\text{For any set of formulae } A_L \\
& \bigvee A_L = \bigvee_{a \in A_L} a =  \bigvee_{k=1}^n a_k = a_1 \vee a_2 \vee .. \vee a_n \\
& \bigwedge A_L = \bigwedge_{a \in A_L} a = \bigwedge_{k=1}^n a_k = a_1 \wedge a_2 \wedge .. \wedge a_n \\
&\text{where } |A_L| = n \text{ and } a_1, a_2 ... a_n \text{ are unique formulae in } A_L
\end{align*}
For any $k$-tuple $T$,  $\pi_i(T)$ would be used to denote the $i$-th element of $T$ where $i \leq k$ e.g.
\begin{align*}
A = (\pi_1(A),\pi_2(A)) \quad
\end{align*}
where $A$ is a 2-tuple (often referred to as an ordered pair).

To avoid ambiguity, for any function $f : X \longrightarrow Y$, we would use $f^{members}$ to denote a new function, from $\mathcal{P}(X)$ to $\mathcal{P}(Y)$, such that
\begin{align*}
&\forall A \in \mathcal{P}(X) \\
&\qquad f^{members}(A) = \bigcup_{a \in A} \{ f(a)\}
\end{align*}
Here is a generalization of it, $f^{members^n}$, defined recursively:
\begin{gather*}
\begin{split}
&f^{members^1} \coloneqq f^{members} \\
&f^{members^n}(x) = \bigcup_{a \in x} \{ f^{members^{n-1}}(a)  \}  \text{ where } n \in \mathbb{N}_{\geq 2}
\end{split}
\end{gather*}

Intuitively, multiset can be viewed as a generalization of set, where the same element can occur multiple times. Two multisets are the same iff both multisets contain the same distinct elements and every distinct element occurs the same number of times in both multisets. More formally, a multiset is defined as an ordered pair $(A,m_{\mathfrak{m}})$ where $m_{\mathfrak{m}} : A \longrightarrow \mathbb{N}_{\geq 1}$ describes the number of occurrences of each element in the multiset, and $A$ is the set of all distinct elements in the multiset. The cardinality of a multiset $(A,m_{\mathfrak{m}})$ is the sum of all $m_{\mathfrak{m}}(x)$ for $x \in A$. Multisets are expressed using square brackets. Here is an example: let $f$ be a function that always outputs 1,
\begin{align*}
&[f(x) : x \in \mathbb{N}_{\geq 1} : x \leq 3] = [1,1,1] = (\{1\},\{(1,3)\}) \\
&\text{as compared to} \\
&\{f(x) : x \in \mathbb{N}_{\geq 1} : x \leq 3\} = \{1\}
\end{align*}

Intuitively, hypergraph can be viewed as a generalization of graph, where an edge can contain any number of vertices. A hypergraph is defined as an ordered pair $(V,E)$ where $E$ is a subset of $\mathcal{P}(V) \setminus \varnothing$. Elements in $V$ are referred to as vertices while elements in $E$ are referred to as edges or hyperedges. A hypergraph is $k$-uniformed when all of its hyperedges have the same cardinality. A graph in the conventional sense can thus be defined as a 2-uniformed hypergraph.

A full vertex coloring on some hypergraph $(V,E)$ is defined as a function, $C : V \longrightarrow cDom$, where $cDom$ is a non-empty subset of $\mathbb{N}$, often referred to as the set of colors. When $|Dom|=2$, we say the coloring is monochromatic. We would use $FullCol(G,cDom)$ to denote the set of all possible full vertex colorings on a hypergraph $G$ of the set of colors $cDom$. That is to say, for any hypergraph $G$ of $n$ vertices and any $cDom$,
\begin{align*}
|FullCol(G,cDom)| = n^{|cDom|}
\end{align*}

The boolean satisfiability problem (SAT) is the problem of determining if there exists some value-assignment for the variables in a propositional logic formula such that it yields $True$ i.e. it is satisfiable.

A formula is referred to as a tautology when there exists no value-assignment for the variables such that it yields $False$ e.g. $a \lor \neg a$.

We say that a formula is in \textit{Disjunctive Normal Form} (DNF) when it is a disjunction of conjunctions. Let $S$ be a set of formuale, a disjunction is a formula that can be expressed as $\bigwedge S$, while a conjunction is a formula that can be expressed as $\bigvee S$.

\section{\textit{Div point set} as a representation for any set of points in general position on an Euclidean plane}

We start off by introducing an object which we would be referring to as \textit{div point set}.

\begin{defn} A  \textit{div point set} is any ordered pair $(P,\Theta_P)$ satisfying
\begin{alignat}{2}
  \label{def1_1}
  &\mathrlap{\lvert\Theta_P\rvert = \binom{\lvert P\rvert}{2} \land P \not= \varnothing} \\[1.5ex]
  \label{def1_2}
  & \forall D_n \in \Theta_P & \quad & \begin{aligned}[t] \renewcommand\arraystretch{1.25}\begin{array}[t]{|@{\hskip0.6em}l} \color{black}
  D_n \text{ is an ordered pair. } \\
  d_n \coloneqq \pi_1(D_n) \\
  \delta_n \coloneqq \pi_2(D_n) \\
  \lvert d_n\rvert = \lvert\delta_n\rvert =2\\
  d_n \in \mathcal{P}(P)\\
  \bigcup \delta_n = P \setminus d_n \\
  \bigcap \delta_n = \varnothing
  \end{array}
  \end{aligned}\\[1.5ex]
  \label{def1_3}
    & \forall D_n, D_m \in \Theta_P &\quad & D_n=D_m \Leftrightarrow  \pi_1(D_n)= \pi_1(D_m)
\end{alignat}
We would be using $\mathscr{DPS}^*$ to denote the class of all ordered pairs satisfying \reff{def1_1}, \reff{def1_2} and \reff{def1_3} i.e. $\mathscr{X}$ is a \textit{div point set} iff $\mathscr{X} \in \mathscr{DPS}^*$.
\end{defn}

\begin{remark}
Since $\lvert\Theta_P\rvert = \binom{\lvert P\rvert}{2}$ and,  for every $D \in \Theta_P$, $|\pi_1(D)| = 2$ and $\pi_1(D) \in \mathcal{P}(P)$, by \reff{def1_3}, we can conclude that
\begin{align}\label{power_set_car_2_as_divider}
\bigcup_{D \in \Theta_P} \pi_1(D) = \{ d \in \mathcal{P}(P): |d| = 2\}
\end{align}
\end{remark}
\begin{prelude}[Axiom 1]
For any $n$ points in general position in $\mathbb{E}^2$, where $n \geq 2$, we can always select 2 arbitrary points and draw a line across them, dividing the remaining $n-2$ points into 2 disjoint sets. We shall refer to these 2 disjoint sets as \textit{divs} produced by a \textit{divider} made up of the 2 points, and the points in the \textit{divs} as \textit{TBD points} (short for \textit{to-be-distributed-among-divs}).

Any set of points $P$ in general position on an Euclidean plane where $\lvert P\rvert \geq 2$ can be represented by some \textit{div point set} $(P,\Theta_P)$: we shall refer to each $D_n \in \Theta_P$ as a \textit{dividon}, to be interpreted as follows:
\begin{equation}
        \begin{aligned}
 &d_n \coloneqq \pi_1(D_n) & \quad & \begin{aligned}[t] \renewcommand\arraystretch{1.25}\begin{array}[t]{|@{\hskip0.6em}l}
   \text{Let the 2 elements in } d_n \text{ be } a, b \\
 a \text{ and } b \text{ represent the 2 points making up the \textit{divider}}   \end{array}  \end{aligned} \\
  &\delta_n \coloneqq \pi_2(D_n) & \quad & \begin{aligned}[t] \renewcommand\arraystretch{1.25}\begin{array}[t]{|@{\hskip0.6em}l}
  \text{Let the 2 elements in } \delta_n \text{ be } div_1, div_2\\
 div_1 \text{ and } div_2 \text{ represent the 2 \textit{divs} produced by the \textit{divider}}  \end{array}  \end{aligned}
 \end{aligned}
 \end{equation}
Basically, each \textit{dividon} describes the relative positions of the corresponding \textit{TBD points} in terms of how they are distributed between the 2 \textit{divs} produced by each \textit{divider}.

\begin{figure}[b]
\centering
\includegraphics[height=5cm]{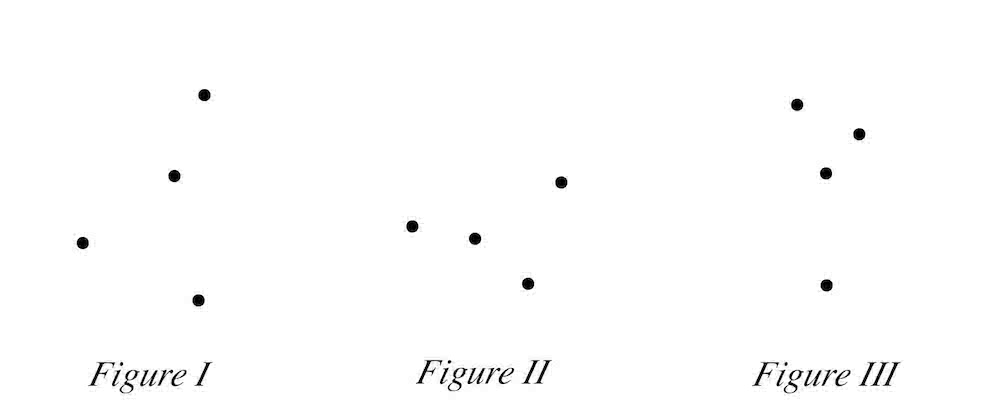}
\end{figure}

The sets of points in \textit{Figures I, II} and \textit{III} can be represented by any \textit{div point set} $(A, \Theta_A)$ as long as $A$ is a set of 4 elements $a$, $b$, $c$, $d$ and
\begin{align*}
    \Theta_{A} = & \{(\{a,b\},\{(\{c\},\{d\}\}), \\
    &\;\; (\{a,c\},\{(\{b\},\{d\}\}), \\
    &\;\; (\{a,d\},\{(\{b\},\{c\}\}), \\
    &\;\; (\{b,c\}, \{(\{a,d\},\varnothing\}), \\
    &\;\; (\{b,d\},\{(\{a,c\},\varnothing\}), \\
    &\;\; (\{c,d\}, \{(\{a,b\},\varnothing\})\})
\end{align*}
To make sense of the above \textit{div point set} representation, we label the third point from the bottom in \textit{Figure I} and the second point from the bottom in \textit{Figures II} and \textit{III} as $a$ (note that each of these points is surrounded by the remaining 3 points in the figure). For the rest of the points in each figure we shall label them arbitrarily as $b$, $c$, and $d$. Notice how in all figures for the 3 \textit{dividers} made up of $a$ and an arbitrary point, we have 2 \textit{divs} of 1 cardinality, and how for the remaining 3 \textit{dividers}, we have rest of the points in a single \textit{div}  - precisely that of what $D \in \Theta_A$ describes.

Only a handful of \textit{div point sets} can be used to represent points in general position in $\mathbb{E}^2$. For majority of $\mathscr{X} \in \mathscr{DPS}^*$, there exists no meaningful interpretation for $\pi_1(\mathscr{X})$ as some set of points in $\mathbb{E}^2$ such that $\pi_2(\mathscr{X})$ describe their relative positions. A classical example would be $(Q, \Theta_Q)$ where $Q$ is any set of 4 elements $a$, $b$, $c$, $d$ and
\begin{align*}
    \Theta_{Q} = & \{(\{a,b\},\{(\{c,d\},\varnothing\}), \\
    &\;\; (\{a,c\},\{(\{b,d\},\varnothing\}), \\
    &\;\; (\{a,d\},\{(\{b,c\},\varnothing\}), \\
    &\;\; (\{b,c\}, \{(\{a,d\},\varnothing\}), \\
    &\;\; (\{b,d\},\{(\{a,c\},\varnothing\}), \\
    &\;\; (\{c,d\}, \{(\{a,b\},\varnothing\})\})
\end{align*}

For any $\mathscr{X} \in \mathscr{DPS}^*$ to have a meaningful interpretation for $\pi_1(\mathscr{X})$ as some set of points in $\mathbb{E}^2$, it has to satisfy certain conditions. After some experimentation with points in $\mathbb{E}^2$, one would make the observation that the following formulae always hold for any distinct points $a$, $b$, $c$, $d$:

\begin{alignat}{2}
 \label{general_position_law1}
  &a\not=b\not=c\not=d \quad \Rightarrow & \quad & \begin{aligned}[t] \renewcommand\arraystretch{1.25}\begin{array}[t]{|@{\hskip0.6em}l} \color{black}
    a \in \langle b,c \rangle^{\scalebox{0.75}[1.0]{-}d} \Leftrightarrow d \in \langle b,c \rangle^{\scalebox{0.75}[1.0]{-}a} \\
  a \in \langle b,c \rangle^d \Leftrightarrow d \in \langle b,c \rangle^a
  \end{array}
  \end{aligned}\\[2.5ex]
  \label{general_position_law2}
  &a\not=b\not=c\not=d \quad \Rightarrow & \quad & \begin{aligned}[t] \renewcommand\arraystretch{1.25}\begin{array}[t]{|@{\hskip0.6em}l} \color{black}
  c \in \langle a,b \rangle^{\scalebox{0.75}[1.0]{-}d} \\
  \quad  \Leftrightarrow ((a \in \langle b,c \rangle^d \land a \in \langle b,d \rangle^c) \\
  \quad  \qquad \lor (a \in \langle b,c \rangle^{\scalebox{0.75}[1.0]{-}d} \land a \in \langle b,d \rangle^{\scalebox{0.75}[1.0]{-}c}) )
  \end{array}
  \end{aligned}\\[2.5ex]
    \label{general_position_law3}
  &a\not=b\not=c\not=d \quad \Rightarrow & \quad & \begin{aligned}[t] \renewcommand\arraystretch{1.25}\begin{array}[t]{|@{\hskip0.6em}l} \color{black}
  c \in \langle a,b \rangle^{d} \\
  \quad  \Leftrightarrow ((a \in \langle b,c \rangle^d \land a \in \langle b,d \rangle^{\scalebox{0.75}[1.0]{-}c}) \\
  \quad\qquad \lor (a \in \langle b,c \rangle^{\scalebox{0.75}[1.0]{-}d} \land a \in \langle b,d \rangle^{c}) )
  \end{array}
  \end{aligned}\\[2.5ex]
      \label{general_position_law4}
  &a\not=b\not=c\not=d \quad \Rightarrow & \quad & \begin{aligned}[t] \renewcommand\arraystretch{1.25}\begin{array}[t]{|@{\hskip0.6em}l} \color{black}
  a \in \langle b,c \rangle^{-d} \land a \in \langle b,d \rangle^{-c} \Rightarrow a \in \langle c,d \rangle^{b}
  \end{array}
  \end{aligned}\end{alignat}
wherein $\langle x,y \rangle^z$ denote the \textit{div} containing $z$ produced by the \textit{divider} made up of the point $x$ and $y$, and $\langle x,y \rangle^{\scalebox{0.75}[1.0]{-}z}$ denote the \textit{div} not containing $z$ produced by the \textit{divider} (here $x$, $y$, and $z$ are metavariable). \reff{general_position_law1} is trivially true. \reff{general_position_law2}, \reff{general_position_law3} and \reff{general_position_law4} are demonstrated in \textit{Figures IV}, \textit{V} and\textit{VI} respectively.

\begin{figure}[p]
\centering
\includegraphics[height=5cm]{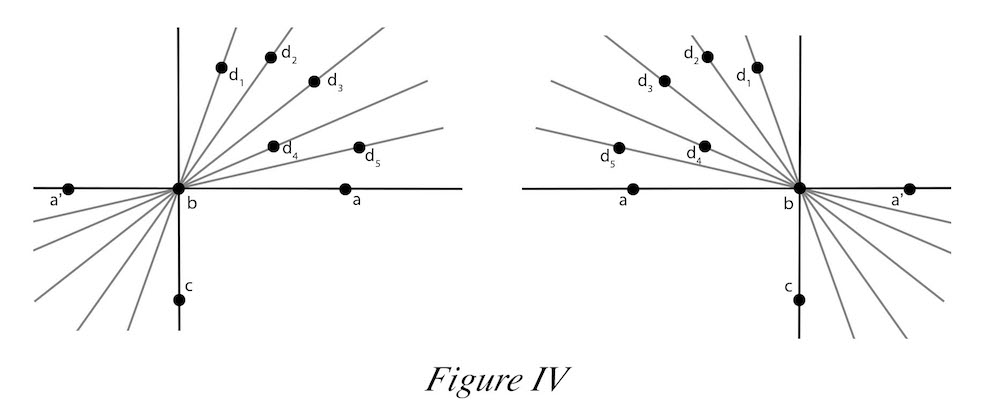}
\end{figure}
\begin{figure}[p]
\centering
 \includegraphics[height=5cm]{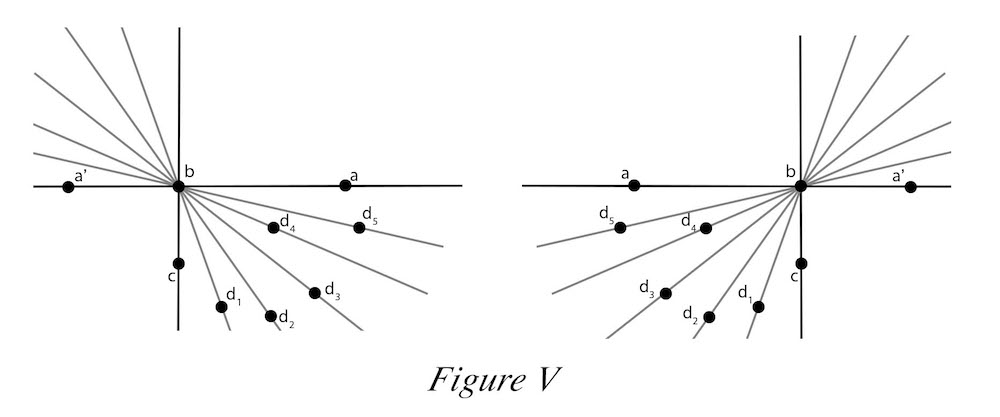}
\end{figure}
\begin{figure}[p]
\centering
\includegraphics[height=5cm]{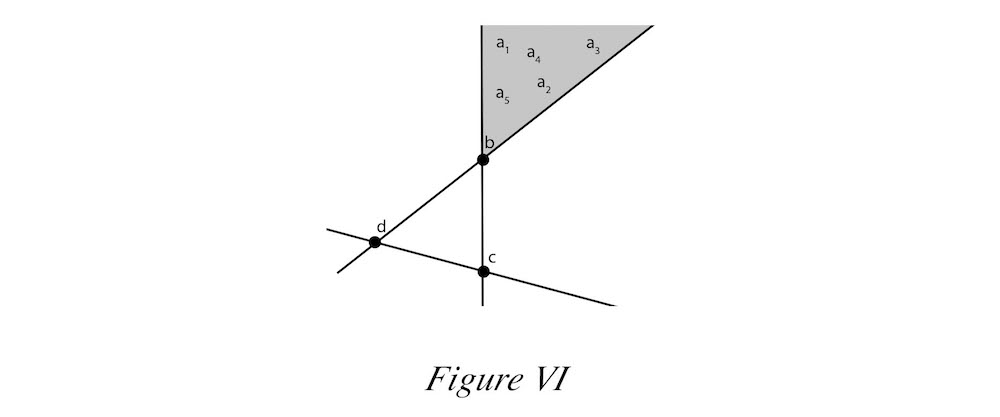}
\end{figure}

In the context of \textit{div point sets},   \reff{general_position_law1} is always true by \reff{def1_2} (recall $ \bigcap \delta = \varnothing$), while \reff{general_position_law2}, \reff{general_position_law3} and \reff{general_position_law4} can each be rewritten as constraints on the \textit{dividons} as shown in \reff{dividon_law1}, \reff{dividon_law2}, and \reff{dividon_law3}.

\begin{align}
\begin{split}
\label{dividon_law1}
&\forall R \in \{ S \in \mathcal{P}(P) : |S| =4 \} \\
&\begin{aligned}\qquad  &\forall D_1, D_2, D_3 \in \Theta_{P} \\
&\qquad \bigcup_{n=1}^{3} \pi_1(D_n) = R \land |\bigcap_{n=1}^{3} \pi_1(D_n)|= 1  \\
&\begin{aligned}\qquad \Rightarrow ( \; &\phi(\pi_2(D_1),R \setminus \pi_1(D_1)) = 0 \\
&\Leftrightarrow \phi(\pi_2(D_2),R  \setminus \pi_1(D_2)) = \phi(\pi_2(D_3),R  \setminus \pi_1(D_3)) \; ) \\
\end{aligned} \end{aligned} \end{split} \\[2.5ex]
\begin{split}\label{dividon_law2}
&\forall R \in \{ S \in \mathcal{P}(P) : |S| =4 \} \\
&\begin{aligned}\qquad  &\forall D_1, D_2, D_3 \in \Theta_{P} \\
&\qquad \bigcup_{n=1}^{3} \pi_1(D_n) = R \land |\bigcap_{n=1}^{3} \pi_1(D_n)|= 1 \\
&\begin{aligned}\qquad \Rightarrow ( \; &\phi(\pi_2(D_1),R \setminus \pi_1(D_1)) = 1 \\
&\Leftrightarrow \phi(\pi_2(D_2),R  \setminus \pi_1(D_2)) \not= \phi(\pi_2(D_3),R  \setminus \pi_1(D_3)) \; ) \\
\end{aligned} \end{aligned} \end{split} \\[2.5ex]
\begin{split}\label{dividon_law3}
&\forall R \in \{ S \in \mathcal{P}(P) : |S| =4 \} \\
&\begin{aligned}\qquad  &\forall D_1, D_2, D_3 \in \Theta_{P} \\
&\qquad \bigcup_{n=1}^{3} \pi_1(D_n) \subset R \land |\bigcup_{n=1}^{3} \pi_1(D_n)| = 3 \land  D_1\not=D_2\not=D_3\\
&\begin{aligned}\qquad \Rightarrow ( \; &\phi(\pi_2(D_1),R \setminus \pi_1(D_1)) = \phi(\pi_2(D_2),R  \setminus \pi_1(D_2)) = 0 \\
&\Rightarrow \phi(\pi_2(D_3),R  \setminus \pi_1(D_3))  = 1 \; ) \\
\end{aligned} \end{aligned} \end{split}
\end{align}
Here $\phi$ determines if two distinct \textit{TBD} points belong to the same \textit{div} i.e.
\begin{equation}\label{phi}
    \phi(\delta,TBD_2)=
\begin{cases}
   1 & \text{if } \;   \exists div \in \delta \; \; |TBD_2 \cap div| = 2 \\
    0 & \text{if }\;  \forall div \in \delta \; \; |TBD_2 \cap div| = 1
\end{cases}
\end{equation}
\end{prelude}
\begin{note}
In \reff{dividon_law1} and \reff{dividon_law2}, it is not necessary to write down $D_1\not=D_2\not=D_3$ explicitly as a part of the conjunction in the antecedent since $\bigcup_{n=1}^{3} \pi_1(D_n) = R \land |\bigcap_{n=1}^{3} \pi_1(D_n)|= 1$ ensures that $D_1$, $D_2$, and $D_3$ are distinct.
\end{note}

\begin{ax} Some $\mathscr{X} \in \mathscr{DPS}^*$ has an interpretation for $\pi_1(\mathscr{X})$ as some set of points in $\mathbb{E}^2$ such that $\pi_2(\mathscr{X})$ describes the relative positions of these points iff $\mathscr{X}$ is in $\mathscr{DPS}^+$, the class of \textit{div point sets} $(P,\Theta_P)$ satisfying \reff{dividon_law1}, \reff{dividon_law2}, and \reff{dividon_law3}.
\end{ax}
\begin{remark}
For \textit{div point sets} of 3 or less points, it is vacuously true that they satisfy \reff{dividon_law1}, \reff{dividon_law2}, and \reff{dividon_law3} and they are all thus in the class $\mathscr{DPS}^+$. This is consistent with Euclidean geometry: any set of 3 points in general position can be represented by any \textit{div point set} of 3 points, and the same goes to any set of 2 or less points.
\end{remark}

\begin{defn} We say that two \textit{div point sets} $(A,\Theta_A)$ and $(B, \Theta_B)$ are isomorphic iff there exists a bijection $f:A \stackrel{\rm{1:1}}{\longrightarrow} B$ preserving the structure of the \textit{dividons}, notationally,
\begin{alignat}{2}\label{isomorphism}
  &(A,\Theta_A) \cong (B,\Theta_B)  && \Leftrightarrow \begin{aligned}[t] \renewcommand\arraystretch{1.25}\begin{array}[t]{@{\hskip0em}l}
  \exists f:A \stackrel{\rm{1:1}}{\longrightarrow} B\\
 \quad\forall D_A \in \Theta_A \\
 \quad\quad\exists D_B \in \Theta_B \\
 \quad\quad\quad f^{members}(\pi_1(D_A)) = \pi_1(D_B)\\
 \quad\quad\quad \Leftrightarrow f^{members^{2}}(\pi_2(D_A)) = \pi_2(D_B) \\
  \end{array}
  \end{aligned}
\end{alignat}
\end{defn}

\begin{remark}It is trivially true that for any two distinct \textit{div point sets} of 3 or less points, they are isomorphic to each other if they are of the same number of points. \end{remark}

\begin{theo} $\neg ( \mathscr{X} \cong Conc_4^1 ) \Leftrightarrow ( \mathscr{X} \cong Conv_4 )$ for all $\mathscr{X} \in \mathscr{DPS}^+_4$ where $\mathscr{DPS}^+_4$ denotes the class of \textit{div point sets} of 4 points in $\mathscr{DPS}^+$ and
\begin{gather}
\begin{split} \label{c4}
    Conc_4^1 = &(Cc_4^1, \Theta_{Cc_4^1}) \\
    Cc_4^1 = &\{1,2,3,4\} \\
    \Theta_{Cc_4^1} = & \{(\{1,2\},\{\{3\},\{4\}\}), \\
    &\;\; (\{1,3\},\{\{2\},\{4\}\}), \\
    &\;\; (\{1,4\},\{\{2\},\{3\}\}), \\
    &\;\; (\{2,3\}, \{\{1,4\},\varnothing\}), \\
    &\;\; (\{2,4\},\{\{1,3\},\varnothing\}), \\
    &\;\; (\{3,4\}, \{\{1,2\},\varnothing\})\}
 \end{split}
\begin{split}
    Conv_4  = &(Cv_4, \Theta_{Cv_4}) \\
    Cv_4 = &\{1,2,3,4\} \\
    \Theta_{Cv_4} = & \{(\{1,2\},\{\{3,4\},\varnothing\}), \\
    &\;\; (\{1,3\},\{\{2\},\{4\}\}), \\
    &\;\; (\{1,4\},\{\{2,3\},\varnothing\}), \\
    &\;\; (\{2,3\}, \{\{1,4\},\varnothing\}), \\
    &\;\; (\{2,4\},\{\{1\},\{3\}\}), \\
    &\;\; (\{3,4\}, \{\{1,2\},\varnothing\})\}
 \end{split}
\end{gather}
\end{theo}
\begin{proof}[Proof for \textit{Theorem 1}]
\begin{summary}In Part 1 of the proof we would define a function $\psi$ that returns 0 or 1 based on the \textit{divs} of a \textit{dividon} of some \textit{div point set} in $\mathscr{DPS}^+_4$. In Part 2 we would define a class $\mathscr{DPS}^\mathbb{N}_4$, a set of vertices for a hypergraph $H$, and a function $Col$ that uses $\psi$, and show that for every $\mathscr{X} \in \mathscr{DPS}^\mathbb{N}_4$, there exists a unique full vertex monochromatic coloring $Col(\pi_2(\mathscr{X}))$ on $H$. In Part 3 we would define a set of edges for $H$ in such a manner that the coloring $Col(\pi_2(\mathscr{X}))$ on $H$ satisfies some conditions iff $\mathscr{X}$ satisfies \reff{dividon_law1} and \reff{dividon_law2}. In Part 4 we would demonstrate that for the coloring to satisfy the conditions, there exists only 3 $\color{black}\mathcal{Scenarios} $. The colorings described in $\color{black}\mathcal{Scenarios}$ II and III are isomorphic$\color{black}^*$  to $Col(\pi_2(Conc_4^1))$ and $Col(\pi_2(Conv_4))$ respectively, and $Conc_4^1$ and $Conv_4$ both satisfy \reff{dividon_law3}, but the \textit{div point set} the coloring described in $\color{black}\mathcal{Scenario}\; I$ is based on does not satisfy \reff{dividon_law3}. Therefore we conclude that \textit{div point sets} of 4 points satisfying \reff{dividon_law1}, \reff{dividon_law2} and \reff{dividon_law3} are either isomorphic to $Conc_4^1$ or $Conv_4$, thus proving \textit{Theorem 1}.
\end{summary}
\begin{proofpart}
Since every \textit{dividon} of any \textit{div point set} in $ \mathscr{DPS}^+_4$ has $4-2=2$ \textit{TBD} points, we can be certain that, let $D$ be a \textit{dividon} and $a$ and $b$ be the \textit{TBD} points, $\pi_2(D) \in \{ type_0, type_1 \}$ where
\begin{gather}
\begin{split}\label{implies_phi_to_delta}
& type_0 \coloneqq \{\{a\},\{b\}\}\\
&type_1 \coloneqq \{\{a,b\},\varnothing\} \\
 \end{split}
\end{gather}
By exploiting the fact every $\pi_2(D)$ is either $type_0$ or $type_1$, for $\mathscr{X} \in \mathscr{DPS}^+_4$, we can define a new function $\psi$, a simpler version of $\phi$ (recall \reff{phi}) that does basically the same thing:
\begin{gather}\label{coloring}\begin{split}
    \psi(\delta)=
\begin{cases}
    0 & \text{if } \;  \forall \textit{div} \in \delta \quad |div| = 1 \\
     1 & \text{if } \; \exists \textit{div} \in \delta \quad |div| = 2
\end{cases}
\end{split}\end{gather}
For every $D \in \pi_2(\mathscr{X})$ where $\mathscr{X} \in \mathscr{DPS}^+_4$, we thus have
\begin{gather}\label{phi_to_delta}\begin{split}
\phi(\pi_2(D), P \setminus \pi_1(D) ) = \psi(\pi_2(D))
\end{split}\end{gather}
\end{proofpart}
\begin{proofpart}
Let $\mathscr{DPS}^\mathbb{N}_4$ be the class of all \textit{div point sets} $(P,\Theta_P)$ for which $P = \{1,2,3,4\}$. All $\mathscr{X} \in \mathscr{DPS}^\mathbb{N}_4$ would have the same set of \textit{dividers} (recall \reff{power_set_car_2_as_divider}). Now let $H=(V,E)$ be a hypergraph whose vertices are the \textit{dividers} of $\mathscr{X} \in \mathscr{DPS}^\mathbb{N}_4$. Using $\psi$, we can define a bijective function $Col$ that transforms the set of \textit{dividons} of any $\mathscr{X} \in \mathscr{DPS}^\mathbb{N}_4$ into some full vertex monochromatic coloring for H.
\begin{equation}
\begin{gathered}\label{coloring1}
Col : \{\pi_2(\mathscr{X}) : \mathscr{X} \in \mathscr{DPS}^\mathbb{N}_4\} \stackrel{\rm{1:1}}{\longrightarrow} FullCol(H,\{0,1\}) \\
Col(\Theta_P) = \{ (\pi_1(D),\psi(\pi_2(D))) : D \in \Theta_P \} \\
\end{gathered}
\end{equation}
It is bijective since $|FullCol(H,\{0,1\})| = |\mathscr{DPS}^\mathbb{N}_4|$ and
\begin{gather}\begin{split}
&\forall \mathscr{X_1}, \mathscr{X_2} \in \mathscr{DPS}^{\mathbb{N}}_4 \\
&\quad Col(\mathscr{\pi_2(X_1)}) =  Col(\mathscr{\pi_2(X_2)}) \Leftrightarrow \mathscr{X_1} = \mathscr{X_2} \\
\end{split}\end{gather}
due to the fact that for any 2 \textit{dividons}, $D_1$ and $D_2$, made up of the same divider, belonging to 2 \textit{div point set} in $\mathscr{DPS}^{\mathbb{N}}_4$ respectively, $\psi(\pi_2(D_1))=\psi(\pi_2(D_2))$  iff $D_1 = D_2$.
\end{proofpart}
\begin{proofpart}
Let any set of 3 \textit{dividers} having 1 point in common to be an edge of $H$ i.e.
\begin{gather}\begin{split}\label{edges}
E \coloneqq \{ e \in \mathcal{P}(V) :  |e| = 3 \land |\bigcap e| = 1 \}
\end{split}\end{gather}
for some $\mathscr{X} \in \mathscr{DPS}^\mathbb{N}_4$ to satisfy \reff{dividon_law1} and \reff{dividon_law2} is equivalent to having $Col(\pi_2(\mathscr{X})) \in FullCol(H,\{0,1\})$ to satisfy I and II:
\begin{enumerate}[I.]
\item For any vertex $v$ colored 0, the other 2 vertices belonging to the same edge as $v$ must be colored the same.
\item For any vertex $v$ colored 1, the other 2 vertices belonging to the same edge as $v$ must be colored differently.
\end{enumerate}
This is in virtue of fact that for any $\mathscr{X} \in \mathscr{DPS}^\mathbb{N}_4$, \reff{dividon_law1} and  \reff{dividon_law2} can be rewritten as having the coloring $C \coloneqq Col(\pi_2(\mathscr{X}))$ to satisfy some formulae, namely \reff{graph_law1} and \reff{graph_law2}.
\begin{align}
\begin{split} \label{graph_law1}
&\forall e \in E \\
&\qquad  \forall d_1, d_2, d_3 \in e \\
&\qquad \qquad  d_1 \not = d_2 \not = d_3 \Rightarrow (C(d_1) = 0 \Leftrightarrow C(d_2) = C(d_3))
\end{split}
\end{align}
\begin{align}
\begin{split}  \label{graph_law2}
&\forall e \in E \\
&\qquad  \forall d_1, d_2, d_3 \in e \\
&\qquad \qquad  d_1 \not = d_2 \not = d_3 \Rightarrow (C(d_1) = 1 \Leftrightarrow C(d_2) \not= C(d_3))
\end{split}
\end{align}
 The above rewriting works because
\begin{align}\begin{split}\label{phi_to_psi}
&\forall R \in \{ S \in \mathcal{P}(P) : |S| =4 \} \\
&\qquad  \forall D \in \Theta_{P} \\
&\qquad \qquad  \phi(\pi_2(D),R \setminus \pi_1(D)) = \phi(\pi_2(D),P \setminus \pi_1(D)) = \psi(\pi_2(D))
\end{split}\end{align}
holds for any \textit{div point set} $(P,\Theta_P)$ for which $|P|=4$, and any \textit{dividons} $D_1, D_2$ and $D_3$ satisfying $|\bigcap_{n=1}^{3} \pi_1(D_n)| = 1 \land  |\bigcup_{n=1}^{3} \pi_1(D_n)| = 4$ would respectively have three \textit{dividers} $d_1$, $d_2$ and $d_3 $ where
\begin{gather}
|\bigcap_{n=1}^{3} d_n| = 1 \land d_1 \not= d_2 \not= d_3
\end{gather}
which are precisely what make up an edge of $H$. Therefore some $\mathscr{X} \in \mathscr{DPS}^\mathbb{N}_4$ satisfies \reff{dividon_law1} and  \reff{dividon_law2} iff $Col(\mathscr{\pi_2(X)})$ satisfies \textit{I} and \textit{II}.
\end{proofpart}
\begin{proofpart}
To satisfy I and II, for every edge of $H$, the 3 vertices it contains must be colored either $[0,0,0]$ or $[0,1,1]$.

Suppose we start off by giving three arbitrary vertices belonging to the same edge the coloring of $[0,0,0]$, by I, the rest of the vertices have to be colored the same (recall that each vertex belongs to 2 different edges). We either end up with $H$ having all vertices colored 0 (let's call it $\color{black}\mathcal{Scenario}\; I$), or 3 vertices colored 0 and 3 vertices colored 1 (let's call it $\color{black}\mathcal{Scenario}\; II$).

Now suppose we start off by giving three arbitrary vertices belonging to the same edge the coloring of $[0,1,1]$. By I, the remaining 2 vertices of another edge the vertex colored 0 belongs to needs to be colored the same. If we color them both 0, the last uncolored vertex of $H$ must then be colored 1 as it belongs to edges wherein both the other 2 vertices are colored differently. We would end up in $\color{black}\mathcal{Scenario}\; II$ again. On the other hand, if we colored them both 1, the last uncolored vertex must then be colored 0 as it belongs to edges wherein both the other 2 vertices are colored the same. Let's call this $\color{black}\mathcal{Scenario}\; III$, where 2 vertices are colored 0 and 4 vertices are colored 1.

A pictorial description of the colorings is shown in Figure VII.

\begin{figure}[!h]
\centering
\includegraphics[height=10cm]{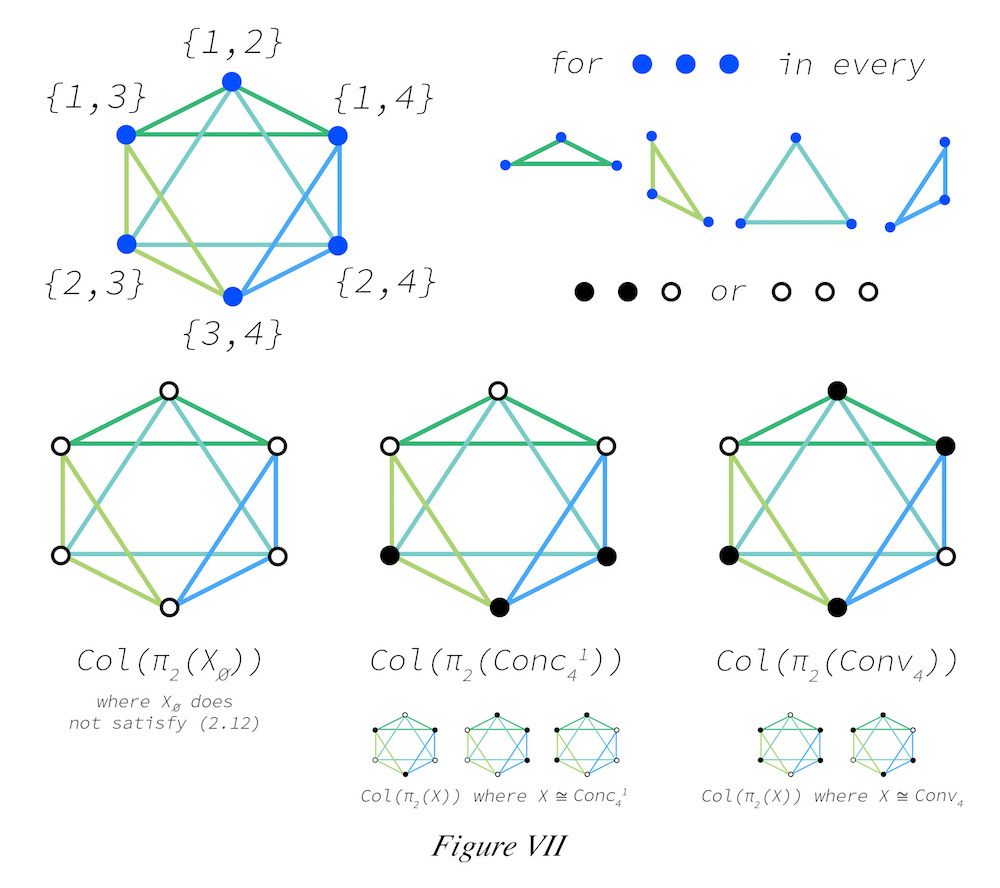}
\end{figure}

$\color{black} \mathcal{Scenario}$ I describes a coloring isomorphic$\color{black}^*$ to $Col(\pi_2(Conc_4^1))$ to $Col(\pi_2(\mathscr{X}_{\varnothing}))$ where $\mathscr{X}_{\varnothing} \in \mathscr{DPS}^{\mathbb{N}}_4$ and
\begin{align*}
    \pi_2(\mathscr{X}_{\varnothing}) = & \{(\{1,2\},\{(\{3,4\},\varnothing\}), \\
    &\;\; (\{1,3\},\{(\{2,4\},\varnothing\}), \\
    &\;\; (\{1,4\},\{(\{2,3\},\varnothing\}), \\
    &\;\; (\{2,3\}, \{(\{1,4\},\varnothing\}), \\
    &\;\; (\{2,4\},\{(\{1,3\},\varnothing\}), \\
    &\;\; (\{3,4\}, \{(\{1,2\},\varnothing\})\})
\end{align*}
while $\color{black}\mathcal{Scenario}$ II describes a coloring isomorphic$\color{black}^*$ to $Col(\pi_2(Conc_4^1))$ and $\color{black} \mathcal{Scenario}$ III descrbies a coloring isomorphic$\color{black}^*$ to $Col(\pi_2(Conv_4))$. $Conc_4^1$ and $Conv_4$ both satisfy \reff{dividon_law3}, and $\mathscr{X}_{\varnothing}$ does not. Since every $\mathscr{X} \in  \mathscr{DPS^\mathbb{*}_4}$ is isomorphic to some $\mathscr{X} \in \mathscr{DPS^\mathbb{N}_4}$, and in $\mathscr{DPS^\mathbb{N}_4}$ only $Conc_4^1$ and $Conv_4$ satisfy all \reff{dividon_law1}, \reff{dividon_law2}, and \reff{dividon_law3}, we conclude that
\begin{gather}\begin{split}
\forall X \in \mathscr{DPS}^{+}_4 \quad \exists a \in \{Conc_4^1,Conv_4 \} \quad X \cong a
\end{split}
\end{gather}
\end{proofpart}

\end{proof}
\begin{note}
\textit{isomorphic$\color{black}^*$}: the isomorphism we are talking about here is that of colorings, which can be defined as follows:
\begin{alignat}{2}
  &C_1 \cong C_2 && \Leftrightarrow \begin{aligned}[t] \renewcommand\arraystretch{1.25}\begin{array}[t]{@{\hskip0em}l}
  \exists f_{C}: \{ \pi_1(c) : c \in C_1\} \stackrel{\rm{1:1}}{\longrightarrow} \{ \pi_1(c) : c \in C_2\}\\
 \qquad\forall c_1 \in C_1 \\
 \qquad\qquad\exists c_2 \in C_2 \\
 \qquad\qquad\qquad f_{C}(\pi_1(c_1)) = \pi_1(c_2) \Rightarrow \pi_2(c_1) = \pi_2(c_2) \\
  \end{array}
  \end{aligned}
\end{alignat}
\end{note}

\begin{remark}
In Euclidean geometry, \textit{Theorem 1} is equivalent as stating that for any set of 4 distinct points in general position, it is either the case that a point can be found inside a triangle formed by connecting the remaining 3 points, or the case that a convex quadrilateral can be created by connecting all 4 points, which can be verified rather easily by a human child with a pen, a piece of paper and a love for geometry.
\end{remark}

\subsection{ \textit{unit div point set} and \textit{sub div point set}}
For \textit{div point sets} of 5 or more points, the function $\psi$  would not be really useful since there would be 3 or more \textit{TBD points} in each \textit{dividon}. That means we cannot use the same approach as above to derive \textit{div point sets} of 5 or more points satisfying \reff{dividon_law1}, \reff{dividon_law2} and \reff{dividon_law3}. With that in mind, we introduce the object \textit{unit div point set}.

\begin{defn}
A \textit{unit div point set} is any ordered pair $(P,\Omega_P)$ satisfying \reff{def3_1}, \reff{def3_2} and \reff{def3_3}.

\begin{alignat}{2}\label{def3_1}
  &\mathrlap{\lvert\Omega_P\rvert = \binom{\lvert P\rvert}{2}  \binom{|P|-2}{2} \land P \not= \varnothing}\\[1.5ex] \label{def3_2}
    & \forall D_n \in \Omega_P & \quad & \begin{aligned}[t] \renewcommand\arraystretch{1.25}\begin{array}[t]{|@{\hskip0.6em}l}
    D_n \text{ is an ordered pair. } \\
  d_n \coloneqq \pi_1(D_n) \\
  \delta_n \coloneqq \pi_2(D_n) \\
    \lvert d_n\rvert = \lvert\delta_n\rvert = \lvert\bigcup \delta_n\rvert =2\\
  d_n \in \mathcal{P}(P)\\
  \bigcup \delta_n \in \mathcal{P}(P \setminus d_n) \\
  \bigcap \delta_n = \varnothing
  \end{array}
  \end{aligned}\\[1.5ex] \label{def3_3}
    & \forall D_n, D_m \in \Omega_P & \quad & \begin{aligned}[t] \renewcommand\arraystretch{1.25}\begin{array}[t]{|@{\hskip0.6em}l}
   \xi(D_n) = \xi(D_m)  \Leftrightarrow D_n=D_m
  \end{array}
  \end{aligned}
\end{alignat}
where
\begin{align}
\xi(D) = \pi_1(D) \cup \bigcup \pi_2(D)
\end{align}
We would be using $\mathscr{UDPS}^*$ to denote the class of all \textit{unit div point set}. Each $D \in \Omega_P$ would be referred to as a \textit{unit dividon}.
\end{defn}
\begin{remark}
Similar to how every \textit{dividon} of $\mathscr{X} \in \mathscr{DPS^*}$ is either $type_0$ or $type_1$ as illustrated in \reff{implies_phi_to_delta}, every \textit{unit dividon} $D$ of any \textit{unit div point set} of points $P$ always satisfies
\begin{gather}
\begin{split}  \label{implies_phi_to_delta2}
& \exists a,b \in P \setminus \pi_1(D) \\
&  \qquad type_0 \coloneqq \{\{a\},\{b\}\}\\
& \qquad type_1 \coloneqq \{\{a,b\},\varnothing\} \\
& \qquad \pi_2(D) \in \{ type_0, type_1 \} \\
 \end{split}
\end{gather}
\end{remark}
\begin{prelude}[Defintion 4]
One may immediately notice that any \textit{div point sets} of 4 points also satisfy \reff{def3_1}, \reff{def3_2} and \reff{def3_3}, similar to how any \textit{unit div point set} of 4 points also satisfy \reff{def1_1}, \reff{def1_2} and \reff{def1_3}, which is to say,
\begin{align}
\{\mathscr{X_{udps}} \in \mathscr{UDPS}^* : |\pi_1(\mathscr{X_{udps}})| = 4 \} = \{\mathscr{X_{dps}} \in \mathscr{DPS}^* : |\pi_1(\mathscr{X_{dps}})| = 4\}
\end{align}
by virtue of the fact that $\binom{4}{2}  \binom{4-2}{2} = \binom{4}{2}$ and
\begin{alignat}{2}
    & \forall \mathscr{X} \in \mathscr{UDPS}^* \notag \\
    &\qquad |\pi_1(\mathscr{X})| = 4 \quad \Rightarrow & \quad & \begin{aligned}[t] \renewcommand\arraystretch{1.25}\begin{array}[t]{|@{\hskip0.6em}l}
  \forall D_n \in \pi_2(\mathscr{X}) \\
   \qquad \bigcup \pi_2(D_n) = P \setminus \pi_1(D_n)  \\
     \forall D_n,D_m \in \pi_2(\mathscr{X}) \\
    \qquad  \pi_1(D_n)= \pi_1(D_m) \Leftrightarrow D_n=D_m
  \end{array}
  \end{aligned}
\end{alignat}
As we can see, the difference between a \textit{div point set} and a \textit{unit div point set} lies in that the former relies on a single \textit{dividon} to describe the distribution of the $|P|-2$ \textit{TBD points} between 2 \textit{divs} for each \textit{divider}, while the later relies on $\binom{|P|-2}{2}$ \textit{unit dividons} for that (since each \textit{unit dividon} only describes the distribution of 2 \textit{TBD points}). For every $\mathscr{X_{dps}} \in \mathscr{DPS}^*$ there exists a unique $\mathscr{X_{udps}} \in \mathscr{UDPS}^*$ which $\mathscr{X_{dps}}$ can be transformed into, by breaking down each \textit{dividon} into $\binom{|P|-2}{2}$ \textit{unit dividons} containing the same \textit{divider}, achievable using the function $\mathscr{bd}$ defined as follows
 \begin{align}
 \begin{split}\label{bd_d_u}
&\mathscr{bd}(D,P) = \{(\pi_1(D),\mathscr{d_u}(\pi_2(D),P_{TBD}) : P_{TBD} \in  \mathcal{P}(P \setminus \pi_1(D)) : |P_{TBD}| = 2\}
 \\
&\mathscr{d_u}(\delta,TBD_2) =
\begin{cases}
   \{P,\varnothing\}& \text{if }\; \phi(\delta,TBD_2) = 1 \\
    \{ x \in \mathcal{P}(P) : |x| = 1  \}& \text{if }\; \phi(\delta,TBD_2) = 0
    \end{cases}
\end{split}
 \end{align}
 $\mathscr{bd}$ takes in a \textit{dividon} and a set of points, and returns a set of \textit{unit dividons}. It makes use of $\mathscr{d_u}$ that takes in a set of \textit{divs} from a \textit{dividon} and a set of 2 points, and returns a set of \textit{divs} for a \textit{unit dividon}.

\end{prelude}
\begin{defn} The function $\undervec{\mathscr{F}}_{\mathscr{udps}}^{\mathscr{DPS}}$ transforms a \textit{div point set} into a \textit{unit div point set}.
\begin{align}\begin{split}\label{def4}
    \undervec{\mathscr{F}}_{\mathscr{udps}}^{\mathscr{DPS}}(\mathscr{X_{dps}}) = (\pi_1(\mathscr{X_{dps}}), \bigcup \{\mathscr{bd}(D,\pi_1(\mathscr{X_{dps}})) : D \in \pi_2(\mathscr{X_{dps}}) \})
\end{split}\end{align}
\\
$\undervec{\mathscr{F}}_{\mathscr{udps}}^{\mathscr{DPS}}$ can be implemented in Haskell as follows:
\begin{lstlisting}
import Control.Monad
import Data.List ((\\))
powerList = filterM (const [True, False])

f:: ([Int],[([Int],[[Int]])]) -> ([Int],[([Int],[[Int]])])
f (points,dividons) = (points,unit_dividons)
    where
        unit_dividons =  foldl (++) [] $ map get_unit_dividons dividons
        get_unit_dividons (d,(delta1:_)) = [(d,(\(a:b:_)->
            if a `in_same_div_as_b` b
                then [[a,b],[]] else [[a],[b]])
            x ) |
            x <- powerList (points \\ d), length x == 2,
            let (in_same_div_as_b) a b = (a `elem` delta1) == (b `elem` delta1)]
\end{lstlisting}

\end{defn}
\begin{remark}
It is no surprise that
\begin{align}
\forall \mathscr{X} \in \mathscr{DPS}^* \qquad \undervec{\mathscr{F}}_{\mathscr{udps}}^{\mathscr{DPS}}(\mathscr{X}) = \mathscr{X} \Leftrightarrow \lvert \pi_1(\mathscr{X}) \rvert = 4
\end{align}
since for each $D\in \mathscr{X}_{4}$ where $\mathscr{X}_{4} \in \mathscr{DPS}_4^*$,  $\mathscr{bd}(D,\pi_1(\mathscr{X_{4}}))$ is a singleton and the one element it contains is $D$.\end{remark}
\begin{remark}
On the other hand,
\begin{align}
\forall \mathscr{X} \in \mathscr{DPS}^* \qquad \undervec{\mathscr{F}}_{\mathscr{udps}}^{\mathscr{DPS}}(\mathscr{X}) = (\pi_1(\mathscr{X}),\varnothing) \Leftrightarrow 1 \leq \lvert \pi_1(\mathscr{X}) \rvert \leq 3
\end{align}
and that is not going to be useful. So it is more sensible to define $\undervec{\mathscr{F}}_{\mathscr{udps}}^{\mathscr{DPS}}$ over \textit{div point sets} of 4 or more points i.e. $\undervec{\mathscr{F}}_{\mathscr{udps}}^{\mathscr{DPS}} : \mathscr{DPS}^*_{\geq4} \longrightarrow \mathscr{UDPS}^*$.
\end{remark}
\begin{lem}
$\undervec{\mathscr{F}}_{\mathscr{udps}}^{\mathscr{DPS}} : \mathscr{DPS}^*_{\geq4} \longrightarrow \mathscr{UDPS}^*$ is injective but not surjective.
\end{lem}

\begin{proof}[Proof for Lemma 1]
It is injective because for every \textit{dividon} $D$ of any \textit{div point set} of 4 or more points, $\mathscr{bd}(D,\pi_1(\mathscr{X_{dps}}))$ in  \reff{def4} differs depending on $D$. By $I$ and $II$ below, we can see that it is not surjective onto the co-domain $\mathscr{UDPS}^*$.
\begin{enumerate}[I.]
\item There exists $\mathscr{X_{udps}} \in \mathscr{UDPS}^*$, where $\undervec{\mathscr{F}}_{\mathscr{udps}}^{\mathscr{DPS}}(\mathscr{W})=\mathscr{X_{udps}}$ iff $\mathscr{W}$ is an ordered pair satisfying all the conditions to be a \textit{div point set} except that,  some \textit{dividon} has more than 2 \textit{divs}, and, as a result, such $\mathscr{W} \not\in \mathscr{DPS}^*$ (Recall $|\delta_n| = 2$ in \reff{def1_2}). E.g. \textit{unit div point sets} with \textit{unit dividons} such as
\begin{align*}\begin{split}
\{(a,b),(\{c\},\{d\})\},\{(a,b),(\{c\},\{e\})\},\{(a,b),(\{e\},\{d\})\}
\end{split}\end{align*}
can only be transformed from a \textit{div-point-set}-like object where $|\pi_2(D)| =3 $ for some \textit{dividion} $D$, in this case: $\{(a,b),(\{c\},\{d\},\{e\})\}$. That is to say, for any $\mathscr{X_{udps}}\prime \in \mathscr{UDPS}^*$, where $\mathscr{X_{udps}}\prime = \undervec{\mathscr{F}}_{\mathscr{udps}}^{\mathscr{DPS}}(\mathscr{X_{dps}})$ for some $\mathscr{X_{dps}} \in \mathscr{DPS}^*$,  $\mathscr{X_{udps}}\prime$ satisfies
\begin{align}\begin{split}\label{2divs}
&\forall D_1,D_2,D_3 \in \pi_2(\mathscr{X_{udps}}\prime) \\
&\qquad D_1 \not = D_2 \not = D_3 \land \pi_1(D_1) = \pi_1(D_2) = \pi_1(D_3) \land |\bigcup_{n=1}^3 \bigcup \pi_2(D_n)| = 3\\
& \qquad \Rightarrow \neg (( \psi(\pi_2(D_1))  = \psi(\pi_2(D_2)) = \psi(\pi_2(D_3)) = 0 )  \\
 \end{split}
\end{align}
\item As a consequence of $\bigcap \delta_n = \varnothing$ in \reff{def1_2}, for any distinct \textit{TBD} points $c$, $d$, and $e$, of some \textit{divider} of a \textit{div point set}, if $c$ and $d$ are in the same \textit{div}, and $d$ and $e$ are in the same \textit{div}, it is certainly the case for $c$ and $e$ to be found in the same \textit{div}. So \textit{unit div point sets} with \textit{unit dividons} such as
\begin{gather*}\begin{split}
\{(a,b),(\{c,d\},\varnothing)\},\{(a,b),(\{c,e\},\varnothing)\},\{(a,b),(\{e\},\{d\})\}
\end{split}\end{gather*}
can not be transformed from any \textit{div point set}. That is to say, for any $\mathscr{X_{udps}}\prime \in \mathscr{UDPS}^*$, where $\mathscr{X_{udps}}\prime = \undervec{\mathscr{F}}_{\mathscr{udps}}^{\mathscr{DPS}}(\mathscr{X_{dps}})$ for some $\mathscr{X_{dps}} \in \mathscr{DPS}^*$,  $\mathscr{X_{udps}}\prime$ satisfies
\begin{align}\begin{split}\label{associativity}
&\forall D_1,D_2,D_3 \in \pi_2(\mathscr{X_{udps}}\prime) \\
&\qquad D_1 \not = D_2 \not = D_3 \land \pi_1(D_1) = \pi_1(D_2) = \pi_1(D_3) \land |\bigcup_{n=1}^3 \bigcup \pi_2(D_n)| = 3\\
& \qquad \Rightarrow \neg ( \psi(\pi_2(D_1))  = \psi(\pi_2(D_2)) = 1 \land  \psi(\pi_2(D_3)) = 0)
 \end{split}
\end{align}
\end{enumerate}

\end{proof}
\begin{remark}
Combining \reff{associativity} and \reff{2divs} above gives \reff{unit_dividon_law0}.
\begin{align}\begin{split}\label{unit_dividon_law0}
&\forall D_1,D_2,D_3 \in \Omega_P \\
&\qquad (D_1 \not = D_2 \not = D_3 \land \pi_1(D_1) = \pi_1(D_2) = \pi_1(D_3) \land |\bigcup_{n=1}^3 \bigcup \pi_2(D_n)| = 3)\\
&\begin{aligned}\qquad  \Rightarrow (&\psi(\pi_2(D_1)) = 1 \Leftrightarrow \psi(\pi_2(D_2)) = \psi(\pi_2(D_3)) ) \\
&\land ( \psi(\pi_2(D_1)) = 0 \Leftrightarrow \psi(\pi_2(D_2)) \not= \psi(\pi_2(D_3)) )
\end{aligned}   \end{split}\end{align}
Let's define $\mathscr{UDPS}^\Theta$ to be a subclass of $\mathscr{UDPS}^*$ for which $\undervec{\mathscr{F}}_{\mathscr{udps}}^{\mathscr{DPS}} : \mathscr{DPS}^*_{\geq4} \longrightarrow \mathscr{UDPS}^{\Theta}$ is bijective. We can be certain that $\mathscr{UDPS}^\Theta \subseteq \mathscr{UDPS}^{\Theta\prime}$, where $\mathscr{UDPS}^{\Theta\prime}$ is the class of \textit{unit div point sets} of 4 or more points satisfying \reff{unit_dividon_law0}. It is likely the case that \reff{unit_dividon_law0} is all that a \textit{unit div point set} must satisfy to be in the class $\mathscr{UDPS}^\Theta$ (i.e. $\mathscr{UDPS}^{\Theta\prime} = \mathscr{UDPS}^{\Theta}$), but that is not important in the current discussion and we would not be going into that.
\end{remark}
\begin{lem}
A \textit{unit div point set} $(P,\Omega_P)$ has an interpretation for $P$ as some set of 4 or more points in $\mathbb{E}^2$ such that $\Omega_P$ describes the relative positions of the points iff it is in $\mathscr{UDPS}^+$ wherein each \textit{unit div point set} satisfies \reff{unit_dividon_law1}, \reff{unit_dividon_law2}, and \reff{unit_dividon_law3}.
\begin{align}
\begin{split}
\label{unit_dividon_law1}
&\forall R \in \{ S \in \mathcal{P}(P) : |S| =4 \} \\
&\begin{aligned}\qquad &\forall D_1, D_2, D_3 \in \Omega_P \\
&\qquad (\; \xi(D_1) =\xi(D_2) =\xi(D_3) = R \land D_1 \not= D_2 \not= D_3 \\
&\qquad \land  |\bigcap_{n=1}^{3} \pi_1(D_n)|= 1 \; )\\
&\begin{aligned}\qquad \Rightarrow ( \; &\psi(\pi_2(D_1)) = 0 \\
&\Leftrightarrow \psi(\pi_2(D_2)) = \psi(\pi_2(D_3)) \; ) \\
\end{aligned} \end{aligned} \end{split}
\end{align}
\begin{align}\begin{split}
\label{unit_dividon_law2}
&\forall R \in \{ S \in \mathcal{P}(P) : |S| =4 \} \\
&\begin{aligned}\qquad  &\forall D_1, D_2, D_3 \in \Omega_P  \\
&\qquad (\; \xi(D_1) =\xi(D_2) =\xi(D_3) = R \land D_1 \not= D_2 \not= D_3 \\
&\qquad  \land  |\bigcap_{n=1}^{3} \pi_1(D_n)|= 1 \;) \\
&\begin{aligned}\qquad \Rightarrow ( \; &\psi(\pi_2(D_1))  = 1 \\
&\Leftrightarrow \psi(\pi_2(D_2))  \not= \psi(\pi_2(D_3))  \; ) \\
\end{aligned} \end{aligned} \end{split}
\end{align}
\begin{align}
\begin{split}\label{unit_dividon_law3}
&\forall R \in \{ S \in \mathcal{P}(P) : |S| =4 \} \\
&\begin{aligned}\qquad &\forall D_1, D_2, D_3 \in \Omega_P \\
&\qquad (\; \xi(D_1) =\xi(D_2) =\xi(D_3) = R \land D_1 \not= D_2 \not= D_3 \\
&\qquad \land |\bigcup_{n=1}^{3} \pi_1(D_n)|= 3 \; )  \\
&\begin{aligned}\qquad \Rightarrow ( \; &\psi(\pi_2(D_1))  = \psi(\pi_2(D_2))  = 0 \\
&\Rightarrow \psi(\pi_2(D_3))   = 1 \; )
\end{aligned} \end{aligned} \end{split}
\end{align}
\end{lem}
\begin{proof}[Proof for Lemma 2] A \textit{div point set} $\mathscr{X_{dps}}$ satisfies \reff{dividon_law1}, \reff{dividon_law2}, and \reff{dividon_law3} iff the \textit{unit div point set}  $\undervec{\mathscr{F}}_{\mathscr{udps}}^{\mathscr{DPS}}(\mathscr{X_{dps}})$ satisfies \reff{unit_dividon_law1}, \reff{unit_dividon_law2}, and \reff{unit_dividon_law3}. Firstly we make the following observation similar to that of \reff{phi_to_psi}: for any \textit{unit divdion} $D_{\mathscr{u}}$ of some \textit{unit div point set} $\mathscr{A_{udps}}$ and its corresponding \textit{divdion} $D$ of the \textit{div point set} $\mathscr{A_{dps}}$ where $\undervec{\mathscr{F}}_{\mathscr{udps}}^{\mathscr{DPS}}(\mathscr{A_{dps}}) = \mathscr{A_{udps}}$ - corresponding in the sense that $D_{u} \in \mathscr{bd}(D,\pi_2(\mathscr{A_{dps}}))$ and so $\pi_1(D_{u}) = \pi_1(D)$ - let $R := \xi(D_\mathscr{u}) $, we would have
\begin{align}\begin{split}\label{D}
&\phi(\pi_2(D_\mathscr{u}),R \setminus \pi_1(D)) = \phi(\pi_2(D_\mathscr{u}),\bigcup \pi_2(D_\mathscr{u})) = \psi(\pi_2(D_\mathscr{u}))
\end{split}\end{align}
By restricting some \textit{unit dividons} $D_1$, $D_2$ and $D_3$ into satisfying $\xi(D_1) = \xi(D_2) = \xi(D_3) = R$ for some set of 4 points $R$, we can replace every occurrence of $\phi(\pi_2(D_n),R \setminus \pi_1(D_n))$ with $\psi(\pi_2(D_n))$ (for $n \in \{1,2,3\}$) in \reff{dividon_law1}, \reff{dividon_law2}, and \reff{dividon_law3}, and ensure the satisfiability of $ \bigcup_{n=1}^{3} \pi_1(D_n) = R$  (in \reff{dividon_law1} and \reff{dividon_law2}) by further restricting these \textit{unit dividons} to be distinct (i.e $D_1 \not= D_2 \not= D_3$). This would give \reff{unit_dividon_law1}, \reff{unit_dividon_law2}, and \reff{unit_dividon_law3}: they are basically a different way of expressing \reff{dividon_law1}, \reff{dividon_law2}, and \reff{dividon_law3} in the case of \textit{unit div point sets}.

Therefore any $\mathscr{X_{udps}} \in \mathscr{UDPS}^+$ has an interpretation for $\pi_1(\mathscr{X_{udps}})$ as some set of 4 or more points in $\mathbb{E}^2$ similar to how any $\mathscr{X_{dps}} \in \mathscr{DPS}^+$ has an interpretation for $\pi_1(\mathscr{X_{dps}})$.

\end{proof}
\begin{note}
In \reff{unit_dividon_law3}, it is not necessary to write down $\bigcup_{n=1}^{3} \pi_1(D_n) \subset R$ explicitly as a part of the conjunction in the antecedent like how it is in \reff{dividon_law3}, since $\xi(D_1) =\xi(D_2) =\xi(D_3) = R \land |\bigcup_{n=1}^{3} \pi_1(D_n)|= 3$ ensures that the union of $\pi_1(D_1)$, $\pi_1(D_2)$, and $\pi_1(D_3)$ is a proper subset of $R$.
\end{note}

\begin{lem}\label{hypergraph-for-all-udps} If $(P,\Omega_P)$ is in $\mathscr{UDPS}^+$, $Col_{\mathscr{udps}}(\Omega_P)$, a full vertex monochromatic coloring on $H_{\mathscr{udps}}$, satisfies \reff{hypergraph_law1} and \reff{hypergraph_law2}. Here $Col_{\mathscr{udps}}$ is a function similar to $Col$ in \reff{coloring1}:
\begin{equation}
\begin{gathered}
Col_{\mathscr{udps}} : \{\pi_2(\mathscr{X}) : \mathscr{X} \in \mathscr{UDPS}^+\} \stackrel{\rm{1:1}}{\longrightarrow} FullCol(H_{\mathscr{udps}},\{0,1\}) \\
Col_{\mathscr{udps}}(\Omega_P)= \{ ((\pi_1(D),\bigcup \pi_2(D)),\psi(\pi_2(D)) : D \in \Omega_P\}
\end{gathered}
\end{equation}
and $H_{\mathscr{udps}}$ is a 3-and-6-uniform hypergraph with 2 sets of hyperedges, $E_1$ and $E_2$, defined as a 3-tuple $H_{\mathscr{udps}} = (V_{\mathscr{udps}},E_1,E_2)$, constructed based on $P$:
\begin{align}
\begin{split}\label{3-6-hypergraph}
V_{\mathscr{udps}} & \coloneqq \bigcup \{ V_{of}(d,P) :  d \in \mathcal{P}(P): |d| = 2\} \\
E_1 & \coloneqq \{ e \in \mathcal{P}(V_{\mathscr{udps}}) :  |e| = 6 \land \forall v_1,v_2 \in e  \; \; \xi(v_1) = \xi(v_2) \} \\
E_2 & \coloneqq \{ e \in \mathcal{P}(V_{\mathscr{udps}}) :  |e| = 3 \land |\bigcup_{v \in e} \pi_2(v)| =  3 \land \forall v_1,v_2 \in e \; \; \pi_1(v_1) = \pi_1(v_2) \} \\
\end{split}\end{align}
with $\xi$ as defined in \textit{Defintion 3} and $V_{of}$ being a function that returns a set of ordered pairs consisting of \textit{divider} and \textit{TBD points} of \textit{unit dividons} of that \textit{divider}, notationally,
\begin{align}
\begin{split}
V_{of}(d,P)= \{ (d, P_{TBD}): P_{TBD} \in \mathcal{P}(P \setminus d) : |P_{TBD}| = 2 \} \\
\end{split}\end{align}
and, finally, we have
\begin{alignat}{2}
&\begin{aligned}
\label{hypergraph_law1}
&\forall e \in E_1\\
&\qquad \exists v_1, v_2 \in e & \quad &  \begin{aligned}[t] \renewcommand\arraystretch{1.25}\begin{array}[t]{|@{\hskip0.6em}l}
v_1 \not = v_2 \\
\pi_1(v_1) = \pi_2(v_2) \\
\pi_1(v_2) = \pi_2(v_1) \\
 C(v_1) = C(v_2) = 0 \\
 C^{members}(e \setminus \{v_1,v_2\}) = \{1\} \\
\end{array}
\end{aligned} \\[1.5ex]
& \qquad \Leftrightarrow \neg \exists v_1, v_2, v_3 \in e & \quad &  \begin{aligned}[t] \renewcommand\arraystretch{1.25}\begin{array}[t]{|@{\hskip0.6em}l}
v_1 \not = v_2  \not= v_3 \\
|\pi_1(v_1) \cap \pi_1(v_2) \cap \pi_1(v_3)| = 1 \\
 C(v_1) = C(v_2) = C(v_3) = 0 \\
 C^{members}(e \setminus \{v_1,v_2,v_3\}) = \{1\}
\end{array}
\end{aligned}
\end{aligned} \\
&\begin{aligned}
\label{hypergraph_law2}
&\forall e \in E_2\\
&\qquad \forall v_1,v_2,v_3 \in e & \quad &  \begin{aligned}[t] \renewcommand\arraystretch{1.25}\begin{array}[t]{|@{\hskip0.6em}l}
v_1 \not= v_2 \not= v_3 \\
\Rightarrow ( C(v_1) = 1 \Leftrightarrow C(v_2) = C(v_3) )\\
\qquad \land ( C(v_1) = 0 \Leftrightarrow C(v_2) \not= C(v_3) )\\
\end{array}
\end{aligned} \\[1.5ex]
\end{aligned}
\end{alignat}
wherein $C = Col_{\mathscr{udps}}(\Omega_P)$.
\end{lem}

\begin{remark}
If $\mathscr{UDPS}^{\Theta\prime} = \mathscr{UDPS}^{\Theta}$, a stronger version of \textit{Lemma \ref{hypergraph-for-all-udps}} is then true: $(P,\Omega_P)$ is in $\mathscr{UDPS}^+$ iff $Col_{\mathscr{udps}}(\Omega_P)$ satisfies \reff{hypergraph_law1} and \reff{hypergraph_law2}.
\end{remark}

\begin{remark} One may notice that the construction of $H_{\mathscr{udps}}$ depends solely on $\pi_1(\mathscr{X_{udps}})$ (i.e. the points of a \textit{unit div point set}), as different from the full vertex coloring, which depends solely on $\pi_2(\mathscr{X_{udps}})$ (i.e. the set of \textit{unit dividons}), similar to how the hypergraph $H$ and its coloring are defined back in the proof for \textit{Theorem 1}. However, the vertices of $H_{\mathscr{udps}}$ are ordered pairs, structurally different from vertices of $H$ which are 2-cardinality sets.  Such definition for the vertices of $H_{\mathscr{udps}}$ in terms of not only the \textit{divider} of a \textit{unit dividon} but also its \textit{TBD points} is necessary. This is because for any \textit{unit div point set} $(P,\Omega_P)$, there exists $\binom{|P|-2}{2}$ distinct \textit{unit dividons} sharing a common \textit{divider}. In order to distinguish \textit{unit dividons} from one another in a \textit{unit div point set} of 5 or more points, we would need to take into account both the \textit{divider} and the \textit{TBD points}.
\end{remark}
\begin{remark} For any \textit{unit div point set} of 4 points, $\mathscr{X_4}$, the second set of edges, $E_2$, of $H_{\mathscr{upds}}$ constructed based on $\pi_1(\mathscr{X_4})$ is an empty set, and thus \reff{hypergraph_law2} is vacuously true for any coloring on such $H_{\mathscr{upds}}$. $E_1$ of such $H_{\mathscr{upds}}$ on the other hand is a singleton. For such $H_{\mathscr{upds}}$, in \reff{hypergraph_law1}, the existential predicate before the logical operator $\Leftrightarrow$ is true iff $\mathscr{X_4}$ is isomorphic to $Conv^4$, while the existential predicate after the logical operator $\neg$ at the right hand side of $\Leftrightarrow$ is true iff $\mathscr{X_4}$ is isomorphic to $Conc^4_1$. Thus for $Col_{\mathscr{udps}}(\pi_2(\mathscr{X_4}))$ to satisfiies \reff{hypergraph_law1} is equivalent to having $\mathscr{X_4}$ isomorphic to either $ Conv^4$ or $Conc^4_1$, which is consistent with \textit{Theorem 1}.
\end{remark}

\begin{proof}[Proof for Lemma 3]
\begin{summary} In Part 1 we show that $Col(\pi_2(\mathscr{X}))$ satisfies \reff{hypergraph_law2} iff $\mathscr{X}$ satisfies \reff{unit_dividon_law0}, and in Part 2 we show that $Col(\pi_2(\mathscr{X}))$ satisfies \reff{hypergraph_law1} iff $\mathscr{X}$ satisfies \reff{unit_dividon_law1}, \reff{unit_dividon_law2}, and \reff{unit_dividon_law3}, for any $\mathscr{X} \in \mathscr{UDPS}^*$.
\end{summary}
\setcounter{proofpart}{0}
\begin{proofpart}
\reff{hypergraph_law2} is simply a different way of expressing \reff{unit_dividon_law0} in the context of coloring: the ordered pairs $(d_n,\bigcup \delta_n)$ of some \textit{unit dividons} $D_n=(d_n,\delta_n)$ that satisfy $(D_1 \not = D_2 \not = D_3 \land \pi_1(D_1) = \pi_1(D_2) = \pi_1(D_3) \land |\bigcup_{n=1}^3 \bigcup \pi_2(D_n)| = 3)$ are defined in \reff{3-6-hypergraph} to be the vertices of an edge in $E_2$.
\end{proofpart}
\begin{proofpart}
The constraints described in \reff{unit_dividon_law1}, \reff{unit_dividon_law2}, and \reff{unit_dividon_law3} revolve around $R$ where $R$ is some 4-cardinality subset of a set of points $P$. For every such $R \subseteq P$, there are a total of $\binom{4}{2}=6$ \textit{unit dividons} $D \in \Theta_P$ where $\xi(D)=R$, for any \textit{unit div point set} $(P,\Theta_P)$. By \textit{Theorem 1}, a \textit{unit div point set} of 4 points (recall that \textit{div point sets} of 4 points are their own \textit{unit div point sets}) satisfies \reff{unit_dividon_law1}, \reff{unit_dividon_law2}, and \reff{unit_dividon_law3} iff it is isomorphic to either $Conc_4^1$ or $Conc_4$. More fundamentally, this means that any \textit{unit div point set} $(P,\Omega_P)$ satisfies \reff{unit_dividon_law1}, \reff{unit_dividon_law2}, and \reff{unit_dividon_law3} iff for every 4-cardinality subset $R$ of $P$,  the 6-cardinality subset $\Omega_{of 6}$ of $\Omega_P$ (where $\xi(D) = R$ for every $D \in \Omega_{of 6}$) is isomorphic$\color{black}^*$ to either $\pi_2(Conc_4^1)$ or $\pi_2(Conv_4)$.

Therefore a \textit{unit div point set} $(P, \Omega_P)$ satisfies \reff{unit_dividon_law1}, \reff{unit_dividon_law2}, and \reff{unit_dividon_law3} iff for every such $R \subseteq P$, let $C'$ be a subset of $Col_{\mathscr{udps}}(\Omega_P)$ where $\xi(\pi_1(c)) = R$ for all $c \in C'$, $C'$ is isomorphic to either $Col(\pi_2(Conc_4^1))$ or $Col(\pi_2(Conv_4))$. Notationally,
\begin{alignat}{2}
\begin{aligned}
&\forall R \in \{ P' \in \mathcal{P}(P) : |P'| = 4 \} \\
&\qquad  C' \coloneqq \{ c \in Col_{\mathscr{udps}}(\Omega_P) : \xi(\pi_1(c))= R\} \\
&\qquad C' \cong Col(\pi_2(Conc_4^1)) \Leftrightarrow \neg ( C' \cong Col(\pi_2(Conv_4)) )
\end{aligned}
\end{alignat}
which is exactly what is expressed in \reff{hypergraph_law1}, considering that
\begin{align}\begin{split}
E_1 = \{ \mathscr{UDs}(R) : R \in \mathcal{P}(P) : |R| = 4 \}
\end{split}
\end{align}
where $\mathscr{UDs}$ returns a set of ordered pairs each consisting of the \textit{divider} and the \textit{TBD} points of every such \textit{unit dividon} for each $R$:
\begin{align}\begin{split}\label{uds}
&\mathscr{UDs}(R) = \{ \mathscr{ud}(d) : d \in \mathcal{P}(R) : |d| = 2 \}\\
&\mathscr{ud}(d) = (d,\ R \setminus d)
\end{split}
\end{align}
\end{proofpart}

\end{proof}
\begin{note}
\textit{isomorphic$\color{black}^*$}: the isomorphism we are talking about here is that of sets of \textit{unit dividons}, which can be defined as follows:
\begin{alignat}{2}
  &\Omega_1 \cong\Omega_2 && \Leftrightarrow \begin{aligned}[t] \renewcommand\arraystretch{1.25}\begin{array}[t]{@{\hskip0em}l}
  |\Omega_1| = |\Omega_2| \\
  \land \exists f_{\Omega}: \bigcup_{D \in \Omega_1} \pi_1(D) \stackrel{\rm{1:1}}{\longrightarrow} \bigcup_{D \in \Omega_2} \pi_1(D)\\
 \qquad\forall D_1 \in \Omega_1 \; \; \exists D_2 \in \Omega_2 \\
 \qquad\qquad f_{\Omega}^{members}(\pi_1(D_1)) = \pi_1(D_2) \Leftrightarrow f_{\Omega}^{members^{2}}(\pi_2(D_1)) = \pi_2(D_2) \\
  \end{array}
  \end{aligned}
\end{alignat}
\end{note}
\begin{defn}
We say that $\mathscr{X_1 \in \mathscr{DPS}^*}$ is a \textit{sub div point set} of  $\mathscr{X_2} \in \mathscr{DPS}^*$  (denoted by $\leq$) iff the set of \textit{unit divdions} of the corresponding \textit{unit div point set} of $\mathscr{X_1}$ is a subset of that of $\mathscr{X_2}$. Notationally,
\begin{gather}\begin{split}\label{sub-div-point-set}
&\forall \mathscr{X_1}, \mathscr{X_2} \in \mathscr{DPS}^*\\
&\quad \mathscr{X_1} \leq \mathscr{X_2} \Leftrightarrow  \pi_2(\undervec{\mathscr{F}}_{\mathscr{udps}}^{\mathscr{DPS}}(\mathscr{X_1})) \subseteq \pi_2(\undervec{\mathscr{F}}_{\mathscr{udps}}^{\mathscr{DPS}}(\mathscr{X_2}))
\end{split}\end{gather}
\end{defn}

\begin{defn}
$\mathscr{Sdps_{of}}$ is a function that returns the set of all \textit{sub div point sets} of $m$ points for some \textit{div point set}, or an empty set depending on $m$.
\begin{align}
\begin{split}
\mathscr{Sdps_{of}}(\mathscr{X_{dps}},m) =
\begin{cases}
   \{ \mathscr{Sdps}(\mathscr{X_{dps}},P_s) : P_s \in \mathcal{P}(\mathcal{\pi_1(\mathscr{X_{dps}})}) : |P_s| = m \} & \text{ if } m \geq 4 \\
   \varnothing & \text{ otherwise}
\end{cases}
\end{split}
\end{align}
where $\mathscr{Sdps}$ returns the \textit{sub div point set} of some set of points $P_s$ of a \textit{div point set}:
\begin{align}
\begin{split}
    &\mathscr{Sdps}(\mathscr{X_{dps}},P_s) = \undervec{\mathscr{F}}_{\mathscr{dps}}^{\mathscr{UDPS}}((P_s,\{ D : D \in \pi_2(\undervec{\mathscr{F}}_{\mathscr{udps}}^{\mathscr{DPS}}(\mathscr{X_{dps}})) : \xi(D) \subseteq P_s  \}) ) \\
    &\text{where } \undervec{\mathscr{F}}_{\mathscr{dps}}^{\mathscr{UDPS}} \text{ is the inverse of } \undervec{\mathscr{F}}_{\mathscr{udps}}^{\mathscr{DPS}} : \mathscr{DPS}^*_{\geq4} \longrightarrow \mathscr{UDPS}^{\Theta}.
\end{split}
\end{align}
Since a \textit{div point set} of $n$ points always has $\binom{n}{m}$ distinct \textit{sub div point sets} of $m$ points and $\undervec{\mathscr{F}}_{\mathscr{udps}}^{\mathscr{DPS}}$ is defined over \textit{div point sets} of 4 or more points, $\mathscr{Sdps_{of}}(\mathscr{X_{dps}},m)$ has the cardinality of $\binom{|\pi_1(\mathscr{X_{dps}})|}{m}$ for all $\mathscr{X_{dps}} \in \mathscr{DPS}^*$ and $m \geq 4$.
\end{defn}
\begin{lem}
For any \textit{div point set} $\mathscr{X}$, let $\mathscr{A}$ and $\mathscr{B}$ be any 2 \textit{sub div point sets} of $\mathscr{X}$, and $k$ be the number of points $\mathscr{A}$ and $\mathscr{B}$ have in common, $\undervec{\mathscr{F}}_{\mathscr{udps}}^{\mathscr{DPS}}(\mathscr{A})$ and $\undervec{\mathscr{F}}_{\mathscr{udps}}^{\mathscr{DPS}}(\mathscr{B})$ always have $6 \binom{k}{4}$ \textit{unit dividons} in common. Notationally,
\begin{align}
\begin{split}
& \forall \mathscr{X} \in \mathscr{DPS^*} \\
& \qquad \forall m, n \in \mathbb{N}_{\geq 1} \\
& \qquad \qquad \forall \mathscr{A} \in \mathscr{Sdps_{of}}(\mathscr{X},m) \;\; \forall \mathscr{B} \in \mathscr{Sdps_{of}}(\mathscr{X},n) \\
& \qquad \qquad \qquad  |\pi_2(\undervec{\mathscr{F}}_{\mathscr{udps}}^{\mathscr{DPS}}(\mathscr{A})) \cap \pi_2(\undervec{\mathscr{F}}_{\mathscr{udps}}^{\mathscr{DPS}}(\mathscr{B}))| = 6 \binom{|\pi_1(\mathscr{A}) \cap \pi_1(\mathscr{B})|}{4}
\end{split}
\end{align}
\end{lem}
\begin{remark}
In the case that both the \textit{sub div point sets} $\mathscr{A}$ and $\mathscr{B}$ are the \textit{div point set} $\mathscr{X}$ itself i.e. $\mathscr{A}=\mathscr{B}=\mathscr{X}$, \textit{Lemma 4} is equivalent to stating that for any \textit{div point set} $\mathscr{X}$,  $\undervec{\mathscr{F}}_{\mathscr{udps}}^{\mathscr{DPS}}(\mathscr{X})$ has $6 \binom{|\pi_1(\mathscr{X})|}{4}$ \textit{unit dividons}, which is true by \reff{def3_1} since $\binom{k}{2} \binom{k-2}{2}=6 \binom{k}{4}$.
\end{remark}
\begin{proof}[Proof for Lemma 4]
For any $m \geq |\pi_1(\mathscr{X})|$ or $m < 4$, the proposition is vacuously true, since  $\mathscr{Sdps_{of}}(\mathscr{X},m)$ would be an empty set. For any $m$ less than  $|\pi_1(\mathscr{X})|$ but greater than or equal to $4$, the proposition can be proven by first observing that $\mathscr{UDs}(R) \cap \mathscr{UDs}(R') = \varnothing  \Leftrightarrow R\not=R'$ (recall \reff{uds}) for any sets $R$ and $R'$ with a cardinality of 4, indicating that no 2 \textit{unit div point sets} of 4 points have in common \textit{unit dividons} of the same \textit{divider} and \textit{TBD points}. Notationally,
 \begin{align}\begin{split} \label{non-cap}
 &\forall \mathscr{A},\mathscr{B} \in \{ \mathscr{X}  : \mathscr{X} \in \mathscr{UDPS}^* : |\pi_1({\mathscr{X}})| = 4 \} \\
&\qquad \qquad \xi^{memebers}(\pi_2(\mathscr{A})) \cap \xi^{memebers}(\pi_2(\mathscr{B})) = \varnothing \Leftrightarrow \pi_1(\mathscr{A}) \not= \pi_1(\mathscr{B})
\end{split}\end{align}
For any 2 \textit{unit div point sets} of 5 or more points, $\mathscr{A_{udps}}$ and $\mathscr{B_{udps}}$, if they have 4 points in common, let the set of such 4 points be $R$ i.e. $ R=\pi_1(\mathscr{A_{udps}})\cap\pi_1(\mathscr{B_{udps}})$, for each $D_\xi \in \mathscr{UDs}(R)$, there exists $D_a \in \mathscr{A_{udps}}$ and  $D_b \in \mathscr{B_{udps}}$ where $\xi(D_a) = \xi(D_b) = D_\xi$. In the case when $\mathscr{A_{udps}}=\undervec{\mathscr{F}}_{\mathscr{udps}}^{\mathscr{DPS}}(\mathscr{A_{dps}})$ and $\mathscr{B_{udps}}=\undervec{\mathscr{F}}_{\mathscr{udps}}^{\mathscr{DPS}}(\mathscr{B_{dps}})$ for some $\mathscr{A_{dps}}$ and $\mathscr{B_{dps}}$ that are both \textit{sub div point sets} of a certain $\mathscr{X} \in \mathscr{DPS}^*$, $D_a = D_b$ for every pair of such \textit{unit dividons} of $\mathscr{A_{udps}}$ and $\mathscr{B_{udps}}$, and thus for every set of 4 points 2 \textit{sub div point set} have in common, they have 6 \textit{unit dividons} in common. Let $k$ be the number of points these 2 \textit{sub div point set} have in common, the number of such distinct sets of 4 points is $\binom{k}{4}$ and thus, by \reff{non-cap}, the number of \textit{unit dividon} they have in common is precisely $6 \binom{|k|}{4}$.

\end{proof}
\begin{theo}
Let $\mathscr{DPS}^+_5$ denote the class of all \textit{div point sets} of 5 points in $\mathscr{DPS}^+$, all $\mathscr{X} \in \mathscr{DPS}^+_5$ either have 4, 2 or 0 distinct \textit{sub div point set} of 4 points isomorphic to $Conc_4^1$ (with the remaining \textit{sub div point sets} of 4 points isomorphic to $Conv_4$).
\end{theo}
\begin{proof}[Proof for Theorem 2]
\begin{summary}
In \textit{Part 1} we show that there exists no $\mathscr{X} \in \mathscr{DPS}^+_5$ where $\mathscr{Sdps_{of}}(\mathscr{X},4)$ has precisely 1, 3 or 5 elements isomorphic to $Conc_4^1$. In \textit{Part 2} we show that there exists $\mathscr{X} \in \mathscr{DPS}^+_5$ where $\mathscr{Sdps_{of}}(\mathscr{X},4)$ has precisely 0, 2 or 4 elements isomorphic to $Conc_4^1$.
\end{summary}
\begin{proofpart}
By \textit{Lemma 2} it is clear that a \textit{div point set} $\mathscr{X_{dps}}$ is in $\mathscr{DPS}^{+}$ iff $\undervec{\mathscr{F}}_{\mathscr{udps}}^{\mathscr{DPS}}(\mathscr{X_{dps}})$ is in $\mathscr{UDPS}^{+}$, which, by \textit{Lemma 3},  implies that $Col_{\mathscr{udps}}(\pi_2(\undervec{\mathscr{F}}_{\mathscr{udps}}^{\mathscr{DPS}}(\mathscr{X_{dps}}))$ satisfies  \reff{hypergraph_law2}. For a full vertex monochromatic coloring on some hypergraph $H_{\mathscr{udps}}$ to satisfy \reff{hypergraph_law2}, every $e \in E_2$ of $H_{\mathscr{udps}}$ must has its vertices colored $[1,0,0]$ or $[1,1,1]$, which mean it would have an even number of vertices colored $0$: 2 multiplying by any number gives an even number and all edges in $ E_2$ of such $H_{\mathscr{udps}}$ are disjoint (for any \textit{unit div point set} of 5 points, there exists exactly $\binom{5-2}{2}=3$ distinct \textit{unit dividon} with the same \textit{divider}).

$Conc_4^1$ has an odd number of \textit{unit dividons} $D$ where $\psi(\pi_2(D)) = 0$, while $Conv_4$ has an even number for such \textit{unit dividons}. By \textit{Lemma 4}, we can see that 2 distinct \textit{sub div point sets} of 4 points always have no \textit{unit dividons} in common. Therefore, for any \textit{div point set} of 5 points $\mathscr{X_{dps}}$, if  $Col_{\mathscr{udps}}(\pi_2(\undervec{\mathscr{F}}_{\mathscr{udps}}^{\mathscr{DPS}}(\mathscr{X_{dps}}))$ satisfies  \reff{hypergraph_law2},  $\mathscr{X_{dps}}$ would not have an odd number of \textit{sub div point sets} of 4 points isomorphic to $Conc_4^1$. We thereby conclude that there exists no $\mathscr{X} \in \mathscr{DPS}^+_5$ where $\mathscr{Sdps_{of}}(\mathscr{X},4)$ has precisely 1, 3 or 5 elements isomorphic to $Conc_4^1$.
\end{proofpart}
\begin{proofpart}
We shall now demonstrate that it is possible to construct \textit{unit div point sets} of 5 points $\mathscr{X_{udps}}$, where $Col_{\mathscr{udps}}(\pi_2(\mathscr{X_{udps}}))$ satisfies \reff{hypergraph_law1} and \reff{hypergraph_law2} and there are precisely 4, 2, or 0 distinct $\Omega_{of 6} \in All_{\Omega_{of 6}}(\mathscr{X_{udps}})$ isomorphic to $\pi_2(Conc_4^1)$, (with the remaining $\Omega_{of 6}$ isomorphic to $Conv_4$), and that such $\mathscr{X_{udps}}$ is in $\mathscr{UDPS}^{\Theta}$. Here $All_{\Omega_{of 6}}$ is a function that returns a set of 6-cardinality sets of \textit{unit dividons} where for any 2 \textit{unit dividons} $D_1$ and $D_2$ in each set, $\xi(D_1)=\xi(D_2)$, defined as follows
\begin{align}
\begin{split}
\label{allOmega}
&All_{\Omega_{of 6}}(\mathscr{X_{udps}}) = \{ \Omega_{of_{based\_on}}(R,\pi_2(\mathscr{X_{udps}})) : R \in \mathcal{P}(\pi_1(\mathscr{X_{udps}})) :  |R| = 4 \} \\
&\Omega_{of_{based\_on}}(R,\Omega) = \{ D \in \Omega : \xi(D) = R \}
\end{split}
\end{align}
\begin{enumerate}[I.]
\item To construct a \textit{unit div point set} $\mathscr{X_{udps}}$ where no $\Omega_{of 6}$ in $All_{\Omega_{of 6}}(\mathscr{X_{udps}})$ is isomorphic to $\pi_2(Conc_4^1)$, we would need to make sure there are only 2 \textit{unit dividons} $D$ in $\Omega_{of 6}$ where $\psi(\pi_2(D)) = 0$ for all 5 $\Omega_{of 6}$ in $All_{\Omega_{of 6}}(\mathscr{X_{udps}})$, notationally
\begin{align}
\begin{split}\label{all_2}
&\forall  \Omega_{of 6}  \in All_{\Omega_{of 6}}(\mathscr{X_{udps}}) \\
&\qquad  |\{D \in \Omega_{of 6} : \psi(\pi_2(D)) = 0 \}| = 2
\end{split}
\end{align}
which can also be expressed as
\begin{align}
\begin{split} \label{all_2_alt}
|\{ \Omega_{of 6}  \in All_{\Omega_{of 6}}(\mathscr{X_{udps}}) :  |\{D \in \Omega_{of 6} : \psi(\pi_2(D)) = 0 \}| = 2 \}| = 5
\end{split}
\end{align}
(Formulating it in terms of the cardinality (instead of using the universal quantifier as in \reff{all_2}) would make things a lot simpler as we go on to II and III.)
\\
Let $D^*$ be the set of all \textit{such unit dividons} i.e.
\begin{align}
\begin{split}
&\forall D \in \pi_2(\mathscr{X_{udps}}) \\
&\qquad D \in D^* \Leftrightarrow \psi(\pi_2(D)) = 0
\end{split}
\end{align}
and $\Omega_{of 6_1}, \Omega_{of 6_2}, \Omega_{of 6_3}, \Omega_{of 6_4}, \Omega_{of 6_5}$ be the 5 elements in $All_{\Omega_{of 6}}(\mathscr{X_{udps}})$, and $D_n^1$ and $D_n^2$ be every 2 such unit divdions in $\Omega_{of 6_n}$, for $n \in \{1,2,3,4,5\}$, i.e.
\begin{align}
\{D_n^1,D_n^2\} = \Omega_{of 6_n} \cap D^*
\end{align}
In order for the coloring to satisfy \reff{hypergraph_law1}, we need to ensure that
\begin{align}\label{same-divider-diving-points}
\begin{split}
\pi_1(D_n^1) =  \bigcup \pi_2(D_n^2) \\
\pi_1(D_n^2) =  \bigcup \pi_2(D_n^1)
\end{split}
\end{align}
holds for $n \in \{1,2,3,4,5\}$. And to satisfy \reff{hypergraph_law2}, we need to ensure that if some \textit{unit dividon} is in $D^*$, we would be able to find another \textit{unit dividon} in $D^*$ that has the same divider, and there exists exactly 1 such \textit{unit dividon}, notationally,
\begin{align}\label{exists-2-dividon-of-same-divider}
\begin{split}
&\forall D' \in D^* \\
&\quad | \{ D \in D^* :   \pi_1(D) = \pi_1(D') \} | = 2
\end{split}
\end{align}
One way to go about achieving that is to let $D_n^1 \not=  D_n^2$ for every $n \in \{1,2,3,4,5\}$ while avoiding $D_n^1 =  D_m^1 \land D_n^2 =  D_m^2$ for any distinct $m$ and $n$. Starting from $D_1$ and going all the way to $D_5$ and we would have
\begin{equation}
\begin{gathered}
\pi_1(D_1^1) = \bigcup \pi_2(D_1^2) = \pi_1(D_2^1) =\bigcup \pi_2(D_2^2) = A\\
\pi_1(D_1^2) = \bigcup \pi_2(D_1^1) = \pi_1(D_3^1) =\bigcup \pi_2(D_3^2) = B\\
\pi_1(D_2^2) = \bigcup \pi_2(D_2^1) = \pi_1(D_4^1) =\bigcup \pi_2(D_4^2) = C\\
\pi_1(D_3^2) = \bigcup \pi_2(D_3^1) = \pi_1(D_6^1) =\bigcup \pi_2(D_6^2) = D \\
\pi_1(D_4^2) = \bigcup \pi_2(D_4^1) = \pi_1(D_5^1) =\bigcup \pi_2(D_5^2) = E\\
\pi_1(D_5^2) = \bigcup \pi_2(D_5^1) = \pi_1(D_6^2) =\bigcup \pi_2(D_6^1) = F
\end{gathered}
\end{equation}
for some 2-cardinality subsets $A,B,C,D,E,F$ of $\pi_1(\mathscr{X_{udps}})$, where
\begin{gather*}
A \not=B \not=C \not=D \not=E \not=F \\
(A \cap B) = (A \cap C) =  (B \cap F) = (C \cap D) = (D \cap E) = (E \cap F) = \varnothing
\end{gather*}
$\mathscr{X_{udps}}$ described above is in $\mathscr{UDPS}^{\Theta}$ because there exists $\mathscr{X_{dps}} \in \mathscr{DPS}^+$ where $\undervec{\mathscr{F}}_{\mathscr{udps}}^{\mathscr{DPS}}(\mathscr{X_{dps}})=\mathscr{X_{udps}}$: such $\mathscr{X_{dps}}$ would be isomorphic to $Conv_5$ defined in \reff{stronger-theorem-2}.
\item To construct a \textit{unit div point set} $\mathscr{X_{udps}}$ where precisely 2 $\Omega_{of 6}$ in $All_{\Omega_{of 6}}(\mathscr{X_{udps}})$ are isomorphic to $\pi_2(Conc_4^1)$ (with the remaining 3 $\Omega_{of 6}$ in $All_{\Omega_{of 6}}(\mathscr{X_{udps}})$ isomorphic to $\pi_2(Conv_4^1)$), we would need to make sure that, for exactly 2 $\Omega_{of 6}$ in $All_{\Omega_{of 6}}(\mathscr{X_{udps}})$, there are precisely 3 \textit{unit dividons} $D$ in $\Omega_{of 6}$ where $\psi(\pi_2(D)) = 1$, and, for the remaining 3 $\Omega_{of 6}$ in $All_{\Omega_{of 6}}(\mathscr{X_{udps}})$, there are precisely 2 \textit{unit dividons} $D$ in $\Omega_{of 6}$ where $\psi(\pi_2(D)) = 0$, notationally
\begin{align}
\begin{split}
& |\{ \Omega_{of 6}  \in All_{\Omega_{of 6}}(\mathscr{X_{udps}}) :  |\{D \in \Omega_{of 6} : \psi(\pi_2(D)) = 0 \}| = 3 \}| = 2 \\
& |\{ \Omega_{of 6}  \in All_{\Omega_{of 6}}(\mathscr{X_{udps}}) :  |\{D \in \Omega_{of 6} : \psi(\pi_2(D)) = 0 \}| = 2 \}| = 3 \\
\end{split}
\end{align}
Using the same notation above, this time we would have $\{D_n^1,D_n^2\} = \Omega_{of 6_n} \cap D^*$ for $n \in \{1,2,3\}$ and $ \{D_n^1,D_n^2,D_n^3\} = \Omega_{of 6_n} \cap D^*$ for $n \in \{4,5\}$. In order for the coloring to satisfy \reff{hypergraph_law1} we need to ensure that \reff{same-divider-diving-points} holds for $n \in \{1,2,3\}$ and
\begin{equation}
\begin{gathered}\label{divider_one_in_common}
|\pi_1(D_n^1) \cap \pi_1(D_n^2) \cap \pi_1(D_n^3)| =1
\end{gathered}
\end{equation}
holds for $n \in \{4,5\}$. And to satisfy \reff{hypergraph_law2}, we also need to ensure that  \reff{exists-2-dividon-of-same-divider} holds as well. One way to go about achieving that is to let $D_4^x$ to have a common \textit{divider} as $D_5^x$ for all $x \in \{1,2\}$, while letting the remaining \textit{unit dividons} in $D_4$ and $D_5$, namely $D_4^3$  and $D_5^3$, to have a common \textit{divider} as $D_1^1$ and $D_2^1$ respectively, and the remaining \textit{unit dividons} in $D_1$ and $D_2$, namely $D_1^2$  and $D_2^2$, to have a common \textit{dividers} as the two \textit{dividons} in $D_3$ respectively.  That is to say, for some subsets of 2 cardinality, $A,B,C,D,E,F$ of $\pi_1(X_{udps})$, we have
\begin{equation}
\begin{gathered}
\pi_1(D_4^1) = \pi_1(D_5^1) = A\\
\pi_1(D_4^2) = \pi_1(D_5^2) = B\\
\pi_1(D_4^3) = \pi_1(D_1^1) = \bigcup \pi_2(D_1^2)  = C\\
\pi_1(D_5^3) = \pi_1(D_2^1) = \bigcup \pi_2(D_2^2)  = D\\
\pi_1(D_1^2) = \bigcup \pi_2(D_1^1) = \pi_1(D_3^1) =\bigcup \pi_2(D_3^2) = E \\
\pi_1(D_2^2) = \bigcup \pi_2(D_2^1) = \pi_1(D_3^2) =\bigcup \pi_2(D_3^1) = F \\
\end{gathered}
\end{equation}
where
\begin{gather*}
A \not=B \not=C \not=D \not=E \not=F \\
|A \cap B \cap C| = 1 \\
|A \cap B \cap D| = 1 \\
(C \cap E) = (D \cap F) = (E \cap F) = \varnothing
\end{gather*}
$\mathscr{X_{udps}}$ described above is in $\mathscr{UDPS}^{\Theta}$ because there exists $\mathscr{X_{dps}} \in \mathscr{DPS}^+$ where $\undervec{\mathscr{F}}_{\mathscr{udps}}^{\mathscr{DPS}}(\mathscr{X_{dps}})=\mathscr{X_{udps}}$: such $\mathscr{X_{dps}}$ would be isomorphic to $Conc_5^1$ defined in \reff{stronger-theorem-2}.
\item To construct a \textit{unit div point set} $\mathscr{X_{udps}}$ where precisely 4 $\Omega_{of 6}$ in $All_{\Omega_{of 6}}(\mathscr{X_{udps}})$ are isomorphic to $\pi_2(Conc_4^1)$ (with the remaining 1 $\Omega_{of 6}$ in $All_{\Omega_{of 6}}(\mathscr{X_{udps}})$ isomorphic to $\pi_2(Conv_4^1)$), this time we would need to make sure that
\begin{align}
\begin{split}
& |\{ \Omega_{of 6}  \in All_{\Omega_{of 6}}(\mathscr{X_{udps}}) :  |\{D \in \Omega_{of 6} : \psi(\pi_2(D)) = 0 \}| = 3 \}| = 4 \\
& |\{ \Omega_{of 6}  \in All_{\Omega_{of 6}}(\mathscr{X_{udps}}) :  |\{D \in \Omega_{of 6} : \psi(\pi_2(D)) = 0 \}| = 2 \}| = 1 \\
\end{split}
\end{align}
Using the same notation above, we would have $\{D_n^1,D_n^2\} = \Omega_{of 6_n} \cap D^*$ for $n \in \{1\}$ and $ \{D_n^1,D_n^2,D_n^3\} = \Omega_{of 6_n} \cap D^*$ for $n \in \{2,3,4,5\}$. In order for the coloring to satisfy \reff{hypergraph_law1}, we need to ensure that \reff{same-divider-diving-points} holds for $n \in \{1\}$ and \reff{divider_one_in_common} holds for $n \in \{2,3,4,5\}$. And to satisfy \reff{hypergraph_law2}, we also need to ensure that  \reff{exists-2-dividon-of-same-divider} holds as well. One way to go about satisfying these conditions is to let $D_4^x$ and $D_2^x$ to have a common \textit{divider} as $D_5^x$ and $D_3^x$ respecitvely, for $x \in \{1,2\}$, while letting the remaining \textit{unit dividons} in $D_2$, $D_4$ and $D_5$, namely $D_2^3$, $D_4^3$  and $D_5^3$, to have a common \textit{divider} as $D_1^2$, $D_3^3$ and $D_1^1$ respectively. That is to say, for some subsets of 2 cardinality, $A,B,C,D,E,F$ of $\pi_1(X_{udps})$, we have
\begin{equation}
\begin{gathered}
\pi_1(D_4^1) = \pi_1(D_5^1) =  A \\
\pi_1(D_4^2) = \pi_1(D_5^2) =  B \\
\pi_1(D_4^3) = \pi_1(D_3^3) =  C\\
\pi_1(D_3^1) = \pi_1(D_2^1) =  D \\
\pi_1(D_3^2) = \pi_1(D_2^2) =  E \\
\pi_1(D_5^3) = \pi_1(D_1^1) = \bigcup \pi_2(D_1^2) =  F\\
\pi_1(D_2^3) = \pi_1(D_1^2) = \bigcup \pi_2(D_1^1) =  G
\end{gathered}
\end{equation}
where
\begin{gather*}
A \not=B \not=C \not=D \not=E \not=F \not= G \\
|A \cap B \cap C| = 1 \\
|A \cap B \cap F| = 1 \\
|C \cap D \cap E | = 1 \\
|D \cap E \cap G | = 1 \\
F \cap G = \varnothing
\end{gather*}
$\mathscr{X_{udps}}$ described above is in $\mathscr{UDPS}^{\Theta}$ because there exists $\mathscr{X_{dps}} \in \mathscr{DPS}^+$ where $\undervec{\mathscr{F}}_{\mathscr{udps}}^{\mathscr{DPS}}(\mathscr{X_{dps}})=\mathscr{X_{udps}}$: such $\mathscr{X_{dps}}$ would be isomorphic to $Conc_5^2$ defined in \reff{stronger-theorem-2}.
\end{enumerate}
\end{proofpart}
\end{proof}
\begin{remark}
A stronger version of \textit{Theorem 2} would state that for all $\mathscr{X_{dps}} \in \mathscr{DPS}^+_5$, $\mathscr{X_{dps}}$ is either isomorphic to $Conv_5$, $Conc_5^1$ or $Conc_5^2$, where
\begin{gather}\label{stronger-theorem-2}
\Scale[0.9]{
\begin{split}
    Conv_5 = &(Cv_5, \Theta_{Cv_5}) \\
    Cv_5 = &\{1,2,3,4,5\} \\
    \Theta_{Cv_5} = & \{(\{1,2\},\{(\{3,4,5\},\varnothing\}), \\
    &\;\; (\{1,3\},\{\{2\},\{4,5\}\}), \\
    &\;\; (\{1,4\},\{\{2,3\},\{5\}\}), \\
    &\;\; (\{1,5\},\{\{2,3,4\},\varnothing\}), \\
    &\;\; (\{2,3\},\{\{1,4,5\},\varnothing\}), \\
    &\;\; (\{2,4\},\{\{1,5\},\{3\}\}), \\
    &\;\; (\{2,5\},\{\{1\},\{3,4\}\}), \\
    &\;\; (\{3,4\}, \{\{1,2,5\},\varnothing\}) \\
    &\;\; (\{3,5\}, \{\{1,2\},\{4\}\}) \\
    &\;\; (\{4,5\}, \{\{1,2,3\},\varnothing\})
 \end{split}
\begin{split}
    Conc_5^1 = &(Cc_5^1, \Theta_{Cv_5}) \\
    Cc_5^1 = &\{1,2,3,4,5\} \\
    \Theta_{Cv_5^1} = & \{(\{1,2\},\{(\{3,4,5\},\varnothing\}), \\
    &\;\; (\{1,3\},\{\{2\},\{4,5\}\}), \\
    &\;\; (\{1,4\},\{\{2,3,5\},\varnothing\}), \\
    &\;\; (\{1,5\},\{\{2,3\},\{4\}), \\
    &\;\; (\{2,3\},\{\{1,4,5\},\varnothing\}), \\
    &\;\; (\{2,4\},\{\{1,5\},\{3\}\}), \\
    &\;\; (\{2,5\},\{\{1\},\{3,4\}\}), \\
    &\;\; (\{3,4\}, \{\{1,2,5\},\varnothing\}) \\
    &\;\; (\{3,5\}, \{\{1,2\},\{4\}\}) \\
    &\;\; (\{4,5\}, \{\{1\},\{2,3\}\})
 \end{split}
 \begin{split}
    Conc_5^2  = &(Cc_5^2, \Theta_{Cc_5^2}) \\
    Cc_5^2 = &\{1,2,3,4,5\} \\
     \Theta_{Cv_5^2} = & \{(\{1,2\},\{(\{3,4,5\},\varnothing\}), \\
    &\;\; (\{1,3\},\{\{2\},\{4,5\}\}), \\
    &\;\; (\{1,4\},\{\{2,3,5\},\varnothing\}), \\
    &\;\; (\{1,5\},\{\{2,3\},\{4\}), \\
    &\;\; (\{2,3\},\{\{1,4\},\{5\}\}), \\
    &\;\; (\{2,4\},\{\{1,3,5\},\varnothing\}), \\
    &\;\; (\{2,5\},\{\{1\},\{3,4\}\}), \\
    &\;\; (\{3,4\}, \{\{1,5\},\{2\}\}) \\
    &\;\; (\{3,5\}, \{\{1,2\},\{4\}\}) \\
    &\;\; (\{4,5\}, \{\{1\},\{2,3\}\}) \end{split}
 }
\end{gather}
To prove this version of \textit{Theorem 2} we would need to prove that there exists no \textit{div point sets} in $\mathscr{DPS}^+_5$ not isomorphic to $Conv_5$, $Conc_5^1$ or $Conc_5^2$.
\end{remark}
\begin{remark} Let $All_{of\Omega}$ be a generalization of $All_{of \Omega_{6}}$ where $All_{of\Omega}(\mathscr{X},4) = All_{of\Omega_{6}}(\mathscr{X})$ i.e.
\begin{align}
All_{of\Omega}(\mathscr{X},n) = \{ \Omega_{of_{based\_on}}(R, \pi_2(\mathscr{X})) : R \in \mathcal{P}(\pi_1(\mathscr{X})) : |R| = n \}
\end{align}
by \textit{Theorem 2}, it is clear that following proposition is false:
\proposition{
A \textit{unit div point set} of 5 or more points $\mathscr{X_{udps}}$ is in $\mathscr{UDPS}^+$ iff all members of $All_{of\Omega_{of 6}}(\mathscr{X_{udps}},n)$ are in $\mathscr{UDPS}^+$, for any $n \in \mathbb{N}_{\geq 4}$ less than $|\pi_1(\mathscr{X_{udps}})|$.
}
However, this weaker version of it still holds true:
\proposition{
If $\mathscr{X_{udps}}$ is in $\mathscr{UDPS}^+$, all members of  $All_{of\Omega_{of 6}}(\mathscr{X_{udps}},n)$ are also in $\mathscr{UDPS}^+$ for any $n \in \mathbb{N}_{\geq 4}$ less than $|\pi_1(\mathscr{X_{udps}})|$.
}
There is undoubtedly some similarity between the false proposition above, and the following proposition which is too false:
\proposition{
A \textit{div point set} of 4 or more points, $\mathscr{X_{dps}}$, is in $\mathscr{DPS}^+$ iff all elements in $\mathscr{Sdps_{of}}(\mathscr{X_{dps}},n)$ are in $\mathscr{DPS}^+$, for any $n \in \mathbb{N}_{\geq 3}$ less than $|\pi_1(\mathscr{X_{dps}})|$.
}
Since it is vacuously true that any \textit{div point sets} of 3 points satisfy \reff{dividon_law1}, \reff{dividon_law2}, and \reff{dividon_law3}, we cannot conclude that a certain \textit{div point set} satisfies \reff{dividon_law1}, \reff{dividon_law2}, and \reff{dividon_law3} just because all its \textit{sub div point sets} of 3 points satisfy them. Now recall \textit{Lemma 3} where $E_2$ of the hypergraph based on $P$ is an empty set in the case when $|P|=4$ and, as a result, for such $E_2$, it is vacuously true that \reff{hypergraph_law2} always holds for any coloring, and thus we cannot conclude that a certain \textit{unit div point set} $\mathscr{X_{udps}}$ where $Col_{\mathscr{udps}}(\pi_2(\mathscr{X_{udps}}))$ satisfies \reff{hypergraph_law2}, just because all memebers of $All_{of\Omega_{of 6}}(\mathscr{X_{udps}},4)$ are isomorphic to some \textit{unit div point set} $\mathscr{A_{udps}}$ where $Col_{\mathscr{udps}}(\pi_2(\mathscr{X_{udps}}))$ satisfies  \reff{hypergraph_law2}.

It can be proven that the proposition regarding \textit{unit div point sets} above is true in the case when $n \in \mathbb{N}_{\geq 5}$, similar to how the proposition regarding \textit{div point sets} is true in the case when $n \in \mathbb{N}_{\geq 4}$.
\end{remark}
\subsection{\textit{convexity}}
The notion that there exists $n$ points forming a convex polygon among some set of points in $\mathbb{E}^2$ can be expressed through \textit{convexity} in the context of \textit{div point sets}.
\begin{defn}
A \textit{div point set} $\mathscr{X}$ has a \textit{convexity} of $n$ iff there exists a \textit{div point set} $\mathscr{X_{sub}}$ such that $\mathscr{X_{sub}}  \leq \mathscr{X} $ and $\mathscr{X_{sub}}$ is isomorphic to $Conv_n$ defined as follow
\begin{align}
\begin{split}\label{conv}
     & Conv_n  = (P,  \{ (d,\delta_{conv}(d,P))  : d \in \mathcal{P}(P) : |d| = 2\}) \\
     &\text{where }
     \begin{cases}
     &P = \{ x \in \mathbb{N}_{\geq 1} : x \leq n \} \\
    &\delta_{conv}(d,P) = \{ \{  p : p \in P : inside(p,d) \}, \{  p : p \in P : outside(p,d) \}\} \\
    &\text{where }
     \begin{cases}
    & inside(p,d) = ( p > min(d) \land p < max(d) ) \\
    & outside(p,d) = ( p < min(d) \lor p > max(d) ) \\
     &\text{where }
     \begin{cases}
    &\text{$min(d)$ returns the smallest number in $d$} \\
    &\text{$max(d)$ returns the biggest number in $d$}.
    \end{cases}
     \end{cases}
     \end{cases}
  \end{split}
  \end{align}
for any $n \in \mathbb{N}_{\geq 3}$. Here is an implementation of $Conv_n$ as a function in Haskell:
\begin{lstlisting}
import Data.List

combine :: Int -> [a] -> [[a]]
combine 0 _  = [[]]
combine n xs = [ y:ys | y:xs' <- tails xs, ys <- combine (n-1) xs']

convex:: Int -> ([Int],[([Int],[[Int]])])
convex n = (points, dividons)
    where
        points = [1..n]
        dividers = combine 2 points
        dividons = [(divider,[div1,div2])
            | divider@(a:b:_) <- dividers,
            let divs = points \\ divider,
            let div1 = [ x | x <- divs, x > a, x < b  ],
            let div2 = divs \\ div1 ]
\end{lstlisting}
\end{defn}
\begin{ax}
For any $\mathscr{X} \in \mathscr{DPS^+}$, $\mathscr{X}$ has an interpretation for $\pi_1(\mathscr{X})$ as some set of points in $E^2$ among which there exists $n$ points forming a convex polygon, iff $\mathscr{X}$ has a convexity of $n$. More precisely, there exists an interpretation for $P' \subseteq \pi_1(\mathscr{X})$ as some set of $|P'|$ points in $E^2$ forming a convex polygon iff $\mathscr{Sdps}(\mathscr{X},P')$ is isomorphic to $Conv_n$, for any $n \geq 3$.
\end{ax}
\begin{remark} One may notice that for $n \geq 4$, all \textit{sub div point sets} of $n-1$ points of $Conv_{n}$ are isomorphic to $Conv_{n-1}$, and as a consequence, a \textit{div point set} with a \textit{convexity} of $k$ would also have a convexity of $m$, for all $k$, $m$ in $\mathbb{N}_{\geq 3}$ where $m < k$. In Euclidean geometry, by \textit{Axiom 2}, that is equivalent to the following proposition: for any $n \geq 4$, after removing any one point from a set of $n$ points that are the vertices of a convex polygon, the remaining points too forms a convex polygon, and as a consequence, any set of points in general position containing $k$ points forming a convex polygon would also contain $m$ points forming a convex polygon, for all $k$, $m$ in $\mathbb{N}_{\geq 3}$ where $m < k$.
\end{remark}
\begin{remark}
We can conclude from \textit{Theorem 2} that a \textit{div point set} of 5 or more points always has a convexity of 4. By \textit{Axiom 2}, this means that we can always find 4 points forming a convex polygon in any set of 5 or more points in general position on an Euclidean plane, as stated in the Erdos-Szekeres conjecture (for the case when $n=4$).
\end{remark}
\section{ A reduction to a \textit{multiset unsatisfiability problem}}
The Erd{\"o}s-Szekeres conjecture can be expressed as a conjunction of \reff{lowerbound} and \reff{upperbound} in the theory of \textit{div point sets}.
\begin{align}
\begin{split} \label{lowerbound}
&\forall n \in \mathbb{N}_{\geq 3} \\
&\qquad \exists \mathscr{A} \in \mathscr{DPS}^+ \quad  |\pi_1(A)| = 2^{n-2} \ \land \forall \mathscr{A_s} \leq \mathscr{A} \quad \mathscr{A_s} \not\cong Conv_n
\end{split} \\
\begin{split} \label{upperbound}
&\forall n \in \mathbb{N}_{\geq 3} \\
&\qquad \forall \mathscr{A} \in \mathscr{DPS}^+ \quad  |\pi_1(A)| > 2^{n-2} \Leftrightarrow \exists \mathscr{A_s} \leq \mathscr{A} \quad \mathscr{A_s} \cong Conv_n
\end{split}
\end{align}
Since the lower bound has been proven to be $2^{n-2}+1$, all is left is to prove \reff{upperbound} and the conjecture would be proven.

\subsection{a combinatorial characteristics of \textit{sub div point sets}}

As we examine \textit{div point sets} of $v$ points for $v > 5$, we would notice this pretty interesting fact about \textit{sub div point sets}: for any natural number $a \geq 1$, let $\mathscr{SDPSS}$ be the set of all \textit{sub div point set} of $v-a$ points of any \textit{div point set} of $v$ points, for any $\mathscr{X_{SDPS}}$ in $\mathscr{SDPSS}$, we can always select $v-a$ distinct $(a+1)$-cardinality subsets of $\mathscr{SDPSS}$, each of which contains $\mathscr{X_{SDPS}}$ and other \textit{div point sets}, and all these \textit{div point sets} that it contains all have $\binom{v-a-1}{t}$ \textit{sub div point sets} of $t$ points in common, for any natural number $t \geq 1$. What is cool about this is that it can be generalized from $a+1$ to $a+b$ for any $b \geq 1$ as long as $a+b$ is smaller than $v$ (and in which case the \textit{div point sets} would have $\binom{v-a-b}{t}$ \textit{sub div point sets} of $t$ points in common). Notationally,
\begin{align}
\begin{split}\label{combinatorics}
&\forall \mathscr{X} \in \mathscr{DPS}^+ \\
&\qquad v \coloneqq |\pi_1(\mathscr{X})| \\
& \qquad \forall a \in \mathbb{N}_{\geq 1} \\
&\qquad \qquad \mathscr{SDPSS} \coloneqq \mathscr{Sdps_{of}}(\mathscr{X}, v-a)\\
&\qquad \qquad \forall  \mathscr{X_{SDPS}}  \in \mathscr{SDPSS} \\
&\qquad \qquad \qquad \forall b \in \{ x: x \in \mathbb{N}_{\geq 1} : x < v-a \}  \\
&\qquad \qquad \qquad \qquad \exists \mathscr{S} \in \mathcal{P_n}( \mathcal{P_n}(\mathscr{SDPSS}, a+b),v-a)\\
&\qquad \qquad \qquad \qquad \qquad \forall \mathscr{s} \in \mathscr{S} \\
&\qquad \qquad \qquad \qquad \qquad \qquad \mathscr{X_{SDPS}} \in \mathscr{s}\\
& \qquad \qquad \qquad \qquad \qquad \qquad  \forall t \in \mathbb{N}_{\geq 1}  \\
&\qquad  \qquad \qquad \qquad \qquad \qquad \qquad |\bigcap_{\mathcal{l} \in \mathscr{s}} \mathscr{Sdps_{of}}(\mathcal{l},t)| =   \binom{v-a-b}{t}\\
\end{split}
\end{align}
where
\begin{align}
\mathcal{P_n}(S,n) = \{ x : x \in \mathcal{P}(S) : |x| = n \}
\end{align}
To understand why such combinatorial characteristic exists, consider this: any 2 \textit{sub div point sets}, $\mathscr{S}_1$ and $\mathscr{S}_2$ of a certain \textit{div point set} is distinct iff they are of distinct points  i.e. $\mathscr{S}_1\not=\mathscr{S}_2 \Leftrightarrow \pi_1(\mathscr{S}_1)\not= \pi_1(\mathscr{S}_2)$, and thus \reff{combinatorics} is equivalent as stating that for any set $\mathcal{N}$ with the same cardinality as $\mathbb{N}$,
\begin{align}
\begin{split}\label{natural-combinatorics}
&\forall X \in \mathcal{P}(\mathcal{N}) \\
&\qquad v \coloneqq |X| \\
& \qquad \forall a \in \mathbb{N}_{\geq 1} \\
&\qquad \qquad X_{subset\_set} \coloneqq \mathcal{P_n}(X, v-a)\\
&\qquad \qquad \forall X_{subset}  \in X_{subset\_set} \\
&\qquad \qquad \qquad \forall b \in \{ x: x \in \mathbb{N}_{\geq 1} : x < v-a \} \\
&\qquad \qquad \qquad \qquad \exists S \in \mathcal{P_n}( \mathcal{P_n}(X_{subset\_set}, a+b),v-a)\\
&\qquad \qquad \qquad \qquad \qquad \forall s \in S \\
&\qquad \qquad \qquad \qquad \qquad \qquad X_{subset} \in s\\
& \qquad \qquad \qquad \qquad \qquad \qquad  \forall n \in \mathbb{N}_{\geq 1}  \\
&\qquad  \qquad \qquad \qquad \qquad \qquad \qquad |\bigcap_{l \in s} \mathcal{P_n}(l,n)| =   \binom{v-a-b}{n}\\
\end{split}
\end{align}
For the purpose of illustration, suppose we have some \textit{div point set} of 9 points $\mathscr{X_9}$, let $isom$ be a bijective function from $\mathscr{Sdps_{of}}(\mathscr{X_9},4)$ to a set of natural numbers $N$ where $N = \{ n: n \in \mathbb{N}_{\geq 1} : n \leq |\mathscr{Sdps_{of}}(\mathscr{X_9},4)|\}$,  the set
\begin{align}
\{ \{ iso(\mathscr{X_4}) : \mathscr{X_4} \in \mathscr{Sdps_{of}}(\mathscr{X_5},4) \} :\mathscr{X_5} \in \mathscr{Sdps_{of}}(\mathscr{X_9},5) \}
\end{align}
shows how \textit{sub div point sets} of 4 points of $\mathscr{X_9}$ (each represented by a distinct natural number) would be disturbed among \textit{sub div point sets} of 5 points of $\mathscr{X_9}$ and is isomorphic$\color{black}^*$ to:
{\tiny
\begin{align*}\begin{split}
&\{\{1,2,7,22,57\},\{1,3,8,23,58\},\{1,4,9,24,59\},\{1,5,10,25,60\},\{1,6,11,26,61\},\{2,3,12,27,62\},\\
&\{2,4,13,28,63\},\{2,5,14,29,64\},\{2,6,15,30,65\},\{3,4,16,31,66\},\{3,5,17,32,67\},\{3,6,18,33,68\},\\
&\{4,5,19,34,69\},\{4,6,20,35,70\},\{5,6,21,36,71\},\{7,8,12,37,72\},\{7,9,13,38,73\},\{7,10,14,39,74\},\\
&\{7,11,15,40,75\},\{8,9,16,41,76\},\{8,10,17,42,77\},\{8,11,18,43,78\},\{9,10,19,44,79\},\{9,11,20,45,80\},\\
&\{10,11,21,46,81\},\{12,13,16,47,82\},\{12,14,17,48,83\},\{12,15,18,49,84\},\{13,14,19,50,85\},\{13,15,20,51,86\},\\
&\{14,15,21,52,87\},\{16,17,19,53,88\},\{16,18,20,54,89\},\{17,18,21,55,90\},\{19,20,21,56,91\},\{22,23,27,37,92\},\\
&\{22,24,28,38,93\},\{22,25,29,39,94\},\{22,26,30,40,95\},\{23,24,31,41,96\},\{23,25,32,42,97\},\{23,26,33,43,98\},\\
&\{24,25,34,44,99\},\{24,26,35,45,100\},\{25,26,36,46,101\},\{27,28,31,47,102\},\{27,29,32,48,103\},\{27,30,33,49,104\},\\
&\{28,29,34,50,105\},\{28,30,35,51,106\},\{29,30,36,52,107\},\{31,32,34,53,108\},\{31,33,35,54,109\},\{32,33,36,55,110\},\\
&\{34,35,36,56,111\},\{37,38,41,47,112\},\{37,39,42,48,113\},\{37,40,43,49,114\},\{38,39,44,50,115\},\{38,40,45,51,116\},\\
&\{39,40,46,52,117\},\{41,42,44,53,118\},\{41,43,45,54,119\},\{42,43,46,55,120\},\{44,45,46,56,121\},\{47,48,50,53,122\},\\
&\{47,49,51,54,123\},\{48,49,52,55,124\},\{50,51,52,56,125\},\{53,54,55,56,126\},\{57,58,62,72,92\},\{57,59,63,73,93\},\\
&\{57,60,64,74,94\},\{57,61,65,75,95\},\{58,59,66,76,96\},\{58,60,67,77,97\},\{58,61,68,78,98\},\{59,60,69,79,99\},\\
&\{59,61,70,80,100\},\{60,61,71,81,101\},\{62,63,66,82,102\},\{62,64,67,83,103\},\{62,65,68,84,104\},\{63,64,69,85,105\},\\
&\{63,65,70,86,106\},\{64,65,71,87,107\},\{66,67,69,88,108\},\{66,68,70,89,109\},\{67,68,71,90,110\},\{69,70,71,91,111\},\\
&\{72,73,76,82,112\},\{72,74,77,83,113\},\{72,75,78,84,114\},\{73,74,79,85,115\},\{73,75,80,86,116\},\{74,75,81,87,117\},\\
&\{76,77,79,88,118\},\{76,78,80,89,119\},\{77,78,81,90,120\},\{79,80,81,91,121\},\{82,83,85,88,122\},\{82,84,86,89,123\},\\
&\{83,84,87,90,124\},\{85,86,87,91,125\},\{88,89,90,91,126\},\{92,93,96,102,112\},\{92,94,97,103,113\},\{92,95,98,104,114\},\\
&\{93,94,99,105,115\},\{93,95,100,106,116\},\{94,95,101,107,117\},\{96,97,99,108,118\},\{96,98,100,109,119\}, \\
&\{97,98,101,110,120\},
\{99,100,101,111,121\},\{102,103,105,108,122\},\{102,104,106,109,123\},\{103,104,107,110,124\},\\
&\{105,106,107,111,125\}, \{108,109,110,111,126\}, \{112,113,115,118,122\},\{112,114,116,119,123\},\{113,114,117,120,124\},\\
&\{115,116,117,121,125\}, \{118,119,120,121,126\},\{122,123,124,125,126\}\}\end{split}
\end{align*}
}Notice how for every $\mathscr{X_5} \in  \mathscr{Sdps_{of}}(\mathscr{X_9},5)$, there exists a set in $\mathcal{P_n}(\mathcal{P_n}(\mathscr{Sdps_{of}}(\mathscr{X_9},5),5),5)$ whose elements are distinct subsets of $\mathscr{Sdps_{of}}(\mathscr{X_9},5)$, each containing $\mathscr{X_5}$ and other \textit{div point sets} all having a common \textit{sub div point set} of 4 points.
\\ \\
We believe that \reff{upperbound} is simply an elegant result of having a structure, whose sub-structures possess the combinatorial characteristic described above, that satisfies a certain constraint, which, in this case, is that described in \reff{assign5} below.

\begin{note}
\textit{isomorphic$\color{black}^*$}: the isomorphism here is defined as a bijective function $f$ from $\mathbb{N}$ to $\mathbb{N}$ such that $f^{members^{2}}(A) = B$.
\end{note}

\subsection{the problem $\color{black}UNSAT_{multiset}^{\mathscr{DPS}^+}$}
\begin{defn}$UNSAT_{multiset}$ is the decision problem of determining if there exists no value-assignment for all variables in $V$, distributed in a certain manner among the multisets in $M$, such that it satisfies the FOL formulae in $C$, where the value-assignment is defined to be a function $Z$: for all $v$ in $V$, $Z(v)= x$ for some $x \in D$. Here $D$, often referred to as the domain, is the set of values a variable can be assigned to. An instance of $UNSAT_{multiset}$ can thus be represented as a 4-tuple $(V,D,M,C)$.
\end{defn}
\begin{prelude}[Definition 9] We shall now present the problem $UNSAT_{multiset}^{\mathscr{DPS}^+}$, a special case of $UNSAT_{multiset}$, of which, if an instance is solved (\textit{solved} in the sense that it is proven that the formulae in $F$ are unsatisfiable), it would prove that, for a particular $n \in \mathbb{N}_{\geq 5}$ (depending on which instance of $UNSAT_{multiset}^{\mathscr{DPS}^+}$ is solved), there exists no \textit{div point set} of $2^{n-2}+1$ points $\mathscr{X}$ that satisfies
\begin{align}
\begin{split}\label{assign5}
&\forall \mathscr{X_5} \in \mathscr{Sdps_{of}}(\mathscr{X},5) \\
&\qquad [Assign(\mathscr{X_4}) : \mathscr{X_4} \in \mathscr{Sdps_{of}}(\mathscr{X_5},4) ] \in \{ [1,1,1,1,0], [1,1,0,0,0], [0,0,0,0,0] \}
\end{split}
\end{align}
but does not satisfy
\begin{align}
\begin{split}\label{n-conv}
& \exists \mathscr{A_s} \in \mathscr{Sdps_{of}}(\mathscr{X},n) \\
& \qquad \forall \mathscr{A_{ss}} \in \mathscr{Sdps_{of}}( \mathscr{A_s} ,4) \quad Assign(\mathscr{A_{ss}}) = 0
\end{split}
\end{align}
where
\begin{align}
\begin{split}
Assign(\mathscr{A}) =
\begin{cases}
   1 & \text{if } \;  \quad   \mathscr{A} \cong Conc_4^1  \\
    0 & \text{if } \;  \quad  \mathscr{A} \cong Conv_4  \\
\end{cases}
\end{split}
\end{align}
consequently proving that the proposition after the universal quantifier in \reff{upperbound} holds for that particular $n$, and thus the $n$-instance of Erd{\"o}s-Szekeres conjecture. This is because if a \textit{div point set} of 5 or more points $\mathscr{X}$ is in $\mathscr{DPS}^+$, by \textit{Theorem 2}, $\mathscr{X}$ satisfy \reff{assign5}. If there exists no \textit{div point set} of $2^{n-2}+1$ points that satisfies \reff{assign5} but not \reff{n-conv}, it would indicate that every \textit{div point set} of $2^{n-2}+1$ or more points in $\mathscr{DPS}^+$ satisfy \reff{n-conv} and therefore has a convexity of $n$. (Note that \reff{upperbound} can be rewritten as follows
\begin{align}\label{upperbound2}
\begin{split}
&\forall n \in \mathbb{N}_{\geq 3} \\
& \qquad\forall \mathscr{A} \in \mathscr{DPS}^+ \\
& \qquad\qquad  |\pi_1(A)| > 2^{n-2} \\
&  \qquad\qquad \Leftrightarrow \exists \mathscr{A_s} \in \mathscr{Sdps_{of}}(\mathscr{A},n) \\
& \qquad \qquad \qquad \forall \mathscr{A_{ss}} \in \mathscr{Sdps_{of}}(\mathscr{A_s},4) \quad Assign(\mathscr{A_{ss}}) = 0
\end{split}
\end{align}
since any \textit{sub div point set} of $k$ points of any $Conv_n$ is isomorphic to $Conv_k$ for all $k,n \in \mathbb{N}_{\geq 3}$ where $n \geq k$).
\end{prelude}
\begin{defn}
$UNSAT_{multiset}^{\mathscr{DPS}^+}$ is a special case, or subproblem, of $UNSAT_{multiset}$ (\textit{subproblem} in the sense that all instances of $UNSAT_{multiset}^{\mathscr{DPS}^+}$ are instances of  $UNSAT_{multiset}$). An instance of $UNSAT_{multiset}$, $(V,D,M,C)$, is an instance of $UNSAT_{multiset}^{\mathscr{DPS}^+}$ iff for some $n \geq 5$,
\begin{equation}
\begin{gathered}
|V| =  \binom{2^{n-2}+1}{4} \\
D = \{0,1\} \\
M = A \cup B \\
\end{gathered}
\end{equation}
and $A$ is a set of 5-cardinality multisets while $B$ is a set of $n$-cardinality multisets, and the variables in $V$ are distributed in $m \in A$ the same way as how elements in $\mathscr{Sdps_{of}}(\mathscr{X},4)$ are distributed in $\mathscr{X_{SPDPS}} \in\mathscr{Sdps_{of}}(\mathscr{X},5)$, while the variables are distributed in $m \in B$ the same way as how elements in $\mathscr{Sdps_{of}}(\mathscr{X},4)$ are distributed in $\mathscr{X_{SPDPS}} \in\mathscr{Sdps_{of}}(\mathscr{X},n)$, where $\mathscr{X}$ is a \textit{div point set} of $2^n+1$ points, and $C$ consists of formulae \ref{constraint1} and \ref{constraint2}.
\begin{align}
\label{constraint1}
&\forall a \in A \qquad a \in \{ [1,1,1,1,0],  [1,1,0,0,0] ,[0,0,0,0,0] \} \\
\label{constraint2}
&\forall b \in B \qquad b \not =  \underbrace{[0,0,0,...,0,0]}_{\binom{n}{4} \; 0's}
\end{align}
The distribution of variables in $A$ and $B$ can be implement in Haskell as follows:
\begin{lstlisting}
import Data.List
import Data.Maybe
type Multiset = [Integer]

merge (a:x) (b:y) = (a,b) : merge x y
merge [] _ = []

choose :: Integer -> Integer -> Integer
n `choose` k
    | k < 0     = 0
    | k > n     = 0
    | otherwise = factorial n `div` (factorial k * factorial (n-k))

factorial :: Integer -> Integer
factorial n =  foldl (*) 1 [1..n]

combine :: Integer -> [Integer] -> [[Integer]]
combine 0 _  = [[]]
combine n xs = [ y:ys | y:xs' <- tails xs, ys <- combine (n-1) xs']

number_of_points = (\n->(2^(n-2)+1))

n_setOf_m_Multisets:: Integer -> Integer -> [Multiset]
n_setOf_m_Multisets m n = [ map fromJust $ map ((flip lookup) encoding)
    (combine 4 m_points) | m_points <- combine n [1..m] ]
    where
        encoding = merge (combine 4 [1..m]) [1..(m `choose` 4)]

setA :: Integer -> [Multiset]
setA n = n_setOf_m_Multisets (number_of_points n) 5

setB :: Integer -> [Multiset]
setB n  = [ x | x <- n_setOf_m_Multisets (number_of_points n) n, 2 `elem` x ]
\end{lstlisting}
\begin{remark} A different implementation may result in a different $M$ for the same $n$. Nonetheless, the different $M$ obtained from a different implementation would be isomorphic to the $M$ obtained from this implementation, in which case we would consider that distribution to be the same. Thus as far as unsatisfiability is concerned, for every $n \in \mathbb{N}_{\geq 5}$, there exists exactly one instance of $UNSAT_{multiset}^{\mathscr{DPS}^+}$.
\end{remark}
\begin{remark}
Each variable in $V$ represents $Assign(\mathscr{X_4})$ for a particular element $\mathscr{X_4} \in \mathscr{Sdps_{of}}(\mathscr{X},4)$ where $\mathscr{X}$ is a \textit{div point set} of $2^{n-2}+1$ points for some $n \in \mathbb{N}_{\geq 5}$. If there exists no value-assignment $Z$ satisfying formulae in $C$, we can be certain that there exists no \textit{div point set} of $2^{n-2}+1$ points $\mathscr{X}$ satisfying \reff{assign5} but not \reff{n-conv} as mentioned above and consequently proving the $n$-instance of conjecture.
\end{remark}
\begin{remark}
Here is the simplest instance of $UNSAT_{multiset}^{\mathscr{DPS}^+}$ (when $n=5$): since $A=B$, we have $|M|=|A|=|B|=\binom{2^{5-2}+1}{5}=126$ multisets, and $|V| = \binom{2^{5-2}+1}{4} = 126$ variables as well (with each denoted by $v_n$ below), distributed among the multisets in $M$ as follows: {\tiny
\begin{align*}\begin{split}
&M = \{[v_{1},v_{2},v_{7},v_{22},v_{57}],[v_{1},v_{3},v_{8},v_{23},v_{58}],[v_{1},v_{4},v_{9},v_{24},v_{59}],[v_{1},v_{5},v_{10},v_{25},v_{60}],[v_{1},v_{6},v_{11},v_{26},v_{61}],[v_{2},v_{3},v_{12},v_{27},v_{62}],\\
&[v_{2},v_{4},v_{13},v_{28},v_{63}],[v_{2},v_{5},v_{14},v_{29},v_{64}],[v_{2},v_{6},v_{15},v_{30},v_{65}],[v_{3},v_{4},v_{16},v_{31},v_{66}],[v_{3},v_{5},v_{17},v_{32},v_{67}],[v_{3},v_{6},v_{18},v_{33},v_{68}],\\
&[v_{4},v_{5},v_{19},v_{34},v_{69}],[v_{4},v_{6},v_{20},v_{35},v_{70}],[v_{5},v_{6},v_{21},v_{36},v_{71}],[v_{7},v_{8},v_{12},v_{37},v_{72}],[v_{7},v_{9},v_{13},v_{38},v_{73}],[v_{7},v_{10},v_{14},v_{39},v_{74}],\\
&[v_{7},v_{11},v_{15},v_{40},v_{75}],[v_{8},v_{9},v_{16},v_{41},v_{76}],[v_{8},v_{10},v_{17},v_{42},v_{77}],[v_{8},v_{11},v_{18},v_{43},v_{78}],[v_{9},v_{10},v_{19},v_{44},v_{79}],[v_{9},v_{11},v_{20},v_{45},v_{80}],\\
&[v_{10},v_{11},v_{21},v_{46},v_{81}],[v_{12},v_{13},v_{16},v_{47},v_{82}],[v_{12},v_{14},v_{17},v_{48},v_{83}],[v_{12},v_{15},v_{18},v_{49},v_{84}],[v_{13},v_{14},v_{19},v_{50},v_{85}],[v_{13},v_{15},v_{20},v_{51},v_{86}],\\
&[v_{14},v_{15},v_{21},v_{52},v_{87}],[v_{16},v_{17},v_{19},v_{53},v_{88}],[v_{16},v_{18},v_{20},v_{54},v_{89}],[v_{17},v_{18},v_{21},v_{55},v_{90}],[v_{19},v_{20},v_{21},v_{56},v_{91}],[v_{22},v_{23},v_{27},v_{37},v_{92}],\\
&[v_{22},v_{24},v_{28},v_{38},v_{93}],[v_{22},v_{25},v_{29},v_{39},v_{94}],[v_{22},v_{26},v_{30},v_{40},v_{95}],[v_{23},v_{24},v_{31},v_{41},v_{96}],[v_{23},v_{25},v_{32},v_{42},v_{97}],[v_{23},v_{26},v_{33},v_{43},v_{98}],\\
&[v_{24},v_{25},v_{34},v_{44},v_{99}],[v_{24},v_{26},v_{35},v_{45},v_{100}],[v_{25},v_{26},v_{36},v_{46},v_{101}],[v_{27},v_{28},v_{31},v_{47},v_{102}],[v_{27},v_{29},v_{32},v_{48},v_{103}],[v_{27},v_{30},v_{33},v_{49},v_{104}],\\
&[v_{28},v_{29},v_{34},v_{50},v_{105}],[v_{28},v_{30},v_{35},v_{51},v_{106}],[v_{29},v_{30},v_{36},v_{52},v_{107}],[v_{31},v_{32},v_{34},v_{53},v_{108}],[v_{31},v_{33},v_{35},v_{54},v_{109}],[v_{32},v_{33},v_{36},v_{55},v_{110}],\\
&[v_{34},v_{35},v_{36},v_{56},v_{111}],[v_{37},v_{38},v_{41},v_{47},v_{112}],[v_{37},v_{39},v_{42},v_{48},v_{113}],[v_{37},v_{40},v_{43},v_{49},v_{114}],[v_{38},v_{39},v_{44},v_{50},v_{115}],[v_{38},v_{40},v_{45},v_{51},v_{116}],\\
&[v_{39},v_{40},v_{46},v_{52},v_{117}],[v_{41},v_{42},v_{44},v_{53},v_{118}],[v_{41},v_{43},v_{45},v_{54},v_{119}],[v_{42},v_{43},v_{46},v_{55},v_{120}],[v_{44},v_{45},v_{46},v_{56},v_{121}],[v_{47},v_{48},v_{50},v_{53},v_{122}],\\
&[v_{47},v_{49},v_{51},v_{54},v_{123}],[v_{48},v_{49},v_{52},v_{55},v_{124}],[v_{50},v_{51},v_{52},v_{56},v_{125}],[v_{53},v_{54},v_{55},v_{56},v_{126}],[v_{57},v_{58},v_{62},v_{72},v_{92}],[v_{57},v_{59},v_{63},v_{73},v_{93}],\\
&[v_{57},v_{60},v_{64},v_{74},v_{94}],[v_{57},v_{61},v_{65},v_{75},v_{95}],[v_{58},v_{59},v_{66},v_{76},v_{96}],[v_{58},v_{60},v_{67},v_{77},v_{97}],[v_{58},v_{61},v_{68},v_{78},v_{98}],[v_{59},v_{60},v_{69},v_{79},v_{99}],\\
&[v_{59},v_{61},v_{70},v_{80},v_{100}],[v_{60},v_{61},v_{71},v_{81},v_{101}],[v_{62},v_{63},v_{66},v_{82},v_{102}],[v_{62},v_{64},v_{67},v_{83},v_{103}],[v_{62},v_{65},v_{68},v_{84},v_{104}],[v_{63},v_{64},v_{69},v_{85},v_{105}],\\
&[v_{63},v_{65},v_{70},v_{86},v_{106}],[v_{64},v_{65},v_{71},v_{87},v_{107}],[v_{66},v_{67},v_{69},v_{88},v_{108}],[v_{66},v_{68},v_{70},v_{89},v_{109}],[v_{67},v_{68},v_{71},v_{90},v_{110}],[v_{69},v_{70},v_{71},v_{91},v_{111}],\\
&[v_{72},v_{73},v_{76},v_{82},v_{112}],[v_{72},v_{74},v_{77},v_{83},v_{113}],[v_{72},v_{75},v_{78},v_{84},v_{114}],[v_{73},v_{74},v_{79},v_{85},v_{115}],[v_{73},v_{75},v_{80},v_{86},v_{116}],[v_{74},v_{75},v_{81},v_{87},v_{117}],\\
&[v_{76},v_{77},v_{79},v_{88},v_{118}],[v_{76},v_{78},v_{80},v_{89},v_{119}],[v_{77},v_{78},v_{81},v_{90},v_{120}],[v_{79},v_{80},v_{81},v_{91},v_{121}],[v_{82},v_{83},v_{85},v_{88},v_{122}],[v_{82},v_{84},v_{86},v_{89},v_{123}],\\
&[v_{83},v_{84},v_{87},v_{90},v_{124}],[v_{85},v_{86},v_{87},v_{91},v_{125}],[v_{88},v_{89},v_{90},v_{91},v_{126}],[v_{92},v_{94},v_{97},v_{103},v_{113}],[v_{92},v_{95},v_{98},v_{104},v_{114}],\\
&[v_{92},v_{93},v_{96},v_{102},v_{112}],[v_{93},v_{94},v_{99},v_{105},v_{115}],[v_{93},v_{95},v_{100},v_{106},v_{116}],[v_{94},v_{95},v_{101},v_{107},v_{117}],[v_{96},v_{97},v_{99},v_{108},v_{118}], \\
&[v_{96},v_{98},v_{100},v_{109},v_{119}],[v_{97},v_{98},v_{101},v_{110},v_{120}],
[v_{99},v_{100},v_{101},v_{111},v_{121}],[v_{102},v_{103},v_{105},v_{108},v_{122}],[v_{102},v_{104},v_{106},v_{109},v_{123}],\\
&[v_{103},v_{104},v_{107},v_{110},v_{124}],[v_{105},v_{106},v_{107},v_{111},v_{125}], [v_{108},v_{109},v_{110},v_{111},v_{126}], [v_{112},v_{113},v_{115},v_{118},v_{122}],[v_{112},v_{114},v_{116},v_{119},v_{123}],\\
&[v_{113},v_{114},v_{117},v_{120},v_{124}],[v_{115},v_{116},v_{117},v_{121},v_{125}], [v_{118},v_{119},v_{120},v_{121},v_{126}],[v_{122},v_{123},v_{124},v_{125},v_{126}]\}
\end{split}
\end{align*}
}%
It is no surprise that the distribution of variables in $m \in M$ above is exactly that of \textit{sub div point sets} of 4 points in $\mathscr{X}_5 \in \mathscr{Sdps_{of}}(\mathscr{X_9},5)$ as shown above.
\end{remark}
\end{defn}

\begin{remark} $UNSAT_{multiset}^{\mathscr{DPS}^+}$ can be reduced into the boolean unsatisfiability problem, the complement of $SAT$, in a rather straightforward manner by first converting each multiset in $A$ into the DNF formula below:
\begin{align}
\bigvee_{v_0 \in \mathcal{A}} (\neg v_0 \land \bigwedge_{v_1 \in \mathcal{A} \setminus \{v_0\}} v_1  ) \lor \bigvee_{\mathcal{A}_{|3|} \in \mathcal{A}_{|3|}^*}( \bigwedge_{v_0 \in \mathcal{A}_3} \neg v_0 \land  \bigwedge_{v_1 \in V \setminus \mathcal{A}_{|3|}} v_1  )  \lor  ( \bigwedge_{v_0 \in \mathcal{A}} \neg v_0  )
\end{align}
where $\mathcal{A}_{|3|}^* = \{ \mathcal{A}_{|3|} \in \mathbb{P}(\mathcal{A}): |\mathcal{A}_{|3|}|=3 \}$ and $\mathcal{A}$ denotes the set of variables in each multiset, and each multiset in $B$ into the DNF formula
 below:\begin{align}
\bigvee_{v \in \mathcal{B}} v
\end{align}
where $\mathcal{B}$ denotes the set of variables in each multiset, then joining all these DNF formulae conjunctively. One may realize that, in the case when $\mathcal{B}=\mathcal{A}$ ,the conjunction of $\bigvee_{v \in \mathcal{B}} v$ and $\bigwedge_{v \in \mathcal{A}} \neg v$ gives a tautology, and thus for the instance of $UNSAT_{multiset}^{\mathscr{DPS}^+}$ where $n=5$, we would have a simpler propositional formula. The same observation can be made in the FOL formulae of such $UNSAT_{multiset}^{\mathscr{DPS}^+}$ instance wherein satisfying both \reff{constraint1}  and \reff{constraint2} is equivalent to satisfying $\forall a \in A \; \; a \in \{ [1,1,1,1,0],  [1,1,0,0,0] \}$.
\end{remark}
\begin{remark}
We thereby conclude that a plausible approach to proving the upper-bound of the Erd{\"o}s-Szekeres conjecture through $UNSAT_{multiset}^{\mathscr{DPS}^+}$ is by induction i.e. we start of by solving the instance of $UNSAT_{multiset}^{\mathscr{DPS}^+}$ where $n=5$ - apparently accomplishable with a modern SAT solver - and then we prove the inductive hypothesis that $\forall m \in \mathbb{N}_{\geq 5} \; \; \mathcal{UNSAT}(m) \Rightarrow \mathcal{UNSAT}(m+1)$ where $\mathcal{UNSAT}(n)$ denotes the unsatisfiability of the $n$-instance of $UNSAT_{multiset}^{\mathscr{DPS}^+}$.
\end{remark}
\begin{remark}
The Erd{\"o}s-Szekeres conjecture would not be disproven even if a certain instance of $UNSAT_{multiset}^{\mathscr{DPS}^+}$ turns out to yield $False$ (i.e. it is satisfiable), since satisfying the constraints only implies that there exists a \textit{div point set} of $2^{n-2}+1$ points for a particular $n \in \mathbb{N}_{\geq 5}$ where
\begin{enumerate}[I.]
\item none of its \textit{sub div point sets} of $n$ points is isomorphic to $Conv_n$
\item each of its \textit{sub div point sets} of 5 points has 4, 2 or 0 distinct \textit{sub div point sets} of 4 points isomorphic to $Conc_4^1$
\end{enumerate}
from which we cannot conclude that such \textit{div point set} is in $\mathscr{DPS}^+$, unless it too satisfies the stronger version of \textit{Theorem 2} i.e. unless proven so, we should not rule out the possibility for some of its \textit{sub div point sets} of 5 points to not be in $\mathscr{DPS}^+$ despite themselves having 4, 2 or 0 distinct \textit{sub div point sets} of 4 points isomorphic to $Conc_4^1$ (with the remaining isomorphic to $Conv_4$).

To disprove the Erd{\"o}s-Szekeres conjecture, not only do we need to show that \reff{upperbound} is false, we need to demonstrate there exists no other constraints besides \reff{dividon_law1}, \reff{dividon_law2}, and \reff{dividon_law3} $\mathscr{X} \in \mathscr{DPS}^*$ has to satisfy such that there exists an interpretation for $\pi_1(\mathscr{X})$ as some set of points in $\mathbb{E}^2$ i.e. \textit{Axiom 1}'s consistency with Euclidean geometry.
\end{remark}

\end{document}